\documentclass[12pt]{article}
\usepackage[margin =1in]{geometry}

\usepackage{amssymb,amsthm,amsmath}
\usepackage{enumerate}
\usepackage{enumitem}
\usepackage{hyperref}
\usepackage{cite}
\usepackage{cleveref}
\usepackage{color}
\usepackage[dvipsnames]{xcolor}

\usepackage{graphicx}
\usepackage{booktabs}
\usepackage{subcaption}

\numberwithin{equation}{section}

\newcommand{\cH}{\mathcal{H}}
\newcommand{\cS}{\mathcal{S}}

\newcommand{\cO}{\mathcal{O}}
\newcommand{\cU}{\mathcal{U}}

\newcommand{\cX}{\mathcal{X}}

\newcommand{\norm}[1] {\left \| #1 \right \|}
\newcommand{\inclu}[0] {\ar@{^{(}->}}

\newcommand{\dist}{{\rm dist}}
\newcommand{\R}{{\mathcal{R}}}

\newcommand{\cE}{\mathcal{E}}

\newcommand{\cM}{\mathcal{M}}

\newcommand{\Vol}{\mathrm{Vol}}
\newcommand{\Unif}{\mathrm{Unif}}
\newcommand{\RR}{\mathbb{R}}
\newcommand{\PP}{\mathbb{P}}

\newcommand{\BB}{\mathbb{B}}

\newcommand{\eps}{\varepsilon_1}
\newcommand{\epsh}{\varepsilon_2}
\newcommand{\Lf}{L_f}
\newcommand{\Lg}{L_{1}}
\newcommand{\Lh}{L_{2}}

\newcommand{\oracle}{G}
\newcommand{\smooth}{g}
\newcommand{\deltac}{\Delta_\smooth}
\newcommand{\deltanc}{\Delta_f}
\newcommand{\x}{y}

\newcommand{\prox}{\mathrm{prox}}
\newcommand{\dom}{\mathrm{dom}\;}
\newcommand{\argmin}{\operatornamewithlimits{argmin}}
\newcommand{\minimize}{\operatornamewithlimits{minimize}}

\newcommand{\Rb}{\RR \cup \{\infty\}}
\newcommand{\epi}{\mathrm{epi}\;}

\newcommand{\Kc}{M}
\newcommand{\Fc}{F}
\newcommand{\Rc}{R}
\newcommand{\tp}{t_\text{pert}}
\newcommand{\alphac}{\alpha}
\newcommand{\betac}{\beta}
\newcommand{\gammac}{\gamma}
\newcommand{\cc}{a}
\newcommand{\ec}{b}
\newcommand{\Ki}{K}

\newcommand{\aprox}{\textsc{ProxOracle}}

\newtheorem{thm}{Theorem}[section]

\newtheorem{proposition}[thm]{Proposition}
\newtheorem{lem}[thm]{Lemma}

\newtheorem{lemma}[thm]{Lemma}

\newtheorem{claim}{Claim}

\theoremstyle{definition}
\newtheorem{definition}[thm]{Definition}
\newtheorem{assumption}{Assumption}

\usepackage{mathtools}
\DeclarePairedDelimiter{\dotp}{\langle}{\rangle}
\usepackage[algo2e, boxruled]{algorithm2e}

\makeatletter
\newcommand{\opnorm}{\@ifstar\@opnorms\@opnorm}
\newcommand{\@opnorms}[1]{%
	\left|\mkern-1.5mu\left|\mkern-1.5mu\left|
	#1
	\right|\mkern-1.5mu\right|\mkern-1.5mu\right|
}
\newcommand{\@opnorm}[2][]{%
	\mathopen{#1|\mkern-1.5mu#1|\mkern-1.5mu#1|}
	#2
	\mathclose{#1|\mkern-1.5mu#1|\mkern-1.5mu#1|}
}
\makeatother

\addtocontents{toc}{\protect\setcounter{tocdepth}{2}}

\begin{document}

	\title{Escaping strict saddle points of the Moreau envelope in nonsmooth optimization}

	\author{Damek Davis\thanks{School of ORIE, Cornell University,
Ithaca, NY 14850, USA;
\texttt{people.orie.cornell.edu/dsd95/}. Research of Davis supported by an Alfred P. Sloan research fellowship and NSF DMS award 2047637.} \qquad Mateo D\'iaz\thanks{CAM, Cornell University. Ithaca, NY 14850, USA;
	\texttt{people.cam.cornell.edu/md825/}} \qquad Dmitriy Drusvyatskiy\thanks{Department of Mathematics, U. Washington,
Seattle, WA 98195; \texttt{www.math.washington.edu/{\raise.17ex\hbox{$\scriptstyle\sim$}}ddrusv}. Research of Drusvyatskiy was supported by the NSF DMS   1651851 and CCF 1740551 awards.}}

\date{}
\maketitle

\begin{abstract}
Recent work has shown that stochastically perturbed gradient methods can efficiently escape strict saddle points of smooth functions. We extend this body of work to nonsmooth optimization, by analyzing an inexact analogue of a stochastically perturbed gradient method applied to the Moreau envelope. The main conclusion is that a variety of algorithms for nonsmooth optimization can escape strict saddle points of the Moreau envelope at a controlled rate. The main  technical insight is that typical algorithms applied to the proximal subproblem yield directions that approximate the gradient of the Moreau envelope in relative terms.
\end{abstract}

\section{Introduction}\label{sec:intro}

Though nonconvex optimization problems are NP hard in general, simple nonconvex optimization techniques, e.g., gradient descent, are broadly used and often highly successful in high-dimensional statistical estimation and machine learning problems.
A common explanation for their success is that \emph{smooth} nonconvex functions $\smooth \colon \RR^d \rightarrow \RR$ found in machine learning have amenable geometry: all local minima are (nearly) global minima and all saddle points are strict (i.e., have a direction of negative curvature).
This explanation is well grounded: several important estimation and learning problems have amenable geometry~\cite{ge2016matrix,sun2015nonconvex,bhojanapalli2016global,ge2017no,sun2018geometric, wang2020efficient}, and simple randomly initialized iterative methods, such as gradient descent, asymptotically avoid strict saddle points~\cite{lee2016gradient, lee2019First}. Moreover, ``randomly perturbed" variants~\cite{jin2019stochastic} ``efficiently" converge to $(\eps, \epsh)$-\emph{approximate second-order critical points}, meaning those satisfying
\begin{equation}
\label{eq:28}
\|\nabla \smooth(x)\| \leq \eps \qquad \text{and} \qquad \lambda_{\min}(\nabla^2 \smooth(x)) \geq -\epsh.
\end{equation}
 Recent work furthermore extends these results to $C^2$ \textit{smooth} manifold constrained optimization~\cite{criscitiello2019efficiently,sun2019escaping, ge2015escaping}. Other extensions to \emph{nonsmooth} convex constraint sets have proposed \emph{second-order} methods for avoiding saddle points, but such methods must \emph{at every step}
		 minimize
	a nonconvex quadratic
		over
		a convex set (an NP hard problem in general) \cite{hallak2020finding,mokhtari2018escaping,nouiehed2018convergence}.

While impressive, the aforementioned works crucially rely on smoothness of objective functions or constraint sets.
This is not an artifact of their proof techniques: there are simple $C^1$ functions for which randomly initialized gradient descent with constant probability converges to points that admit directions of second order descent~\cite[Figure 1]{Davis2019ProximalMA}. 
Despite this example, recent work \cite{Davis2019ProximalMA} shows that randomly initialized \emph{proximal methods} avoid certain ``active" strict saddle points of (nonsmooth) \emph{weakly convex} functions.
The class of weakly convex functions is broad, capturing, for example those
		formed by
	composing
	convex functions $h$
		with
		smooth nonlinear maps $c$,
which often appear in statistical recovery problems.
They moreover show that for ``generic'' semialgebraic problems, every critical point is either a local minimizer or an active strict saddle.
A key limitation of~\cite{Davis2019ProximalMA}, however, is that the result is asymptotic, and in fact pure proximal methods may take exponentially many iterations to find local minimizers~\cite{du2017gradient}.
Motivated by \cite{Davis2019ProximalMA}, the recent work~\cite{huang2021escaping} develops efficiency estimates for certain randomly perturbed proximal methods. The work \cite{huang2021escaping} has two limitations: its measure of complexity appears to be algorithmically dependent and the results do not extend to subgradient methods.

The purpose of this paper is to develop ``efficient" methods for escaping saddle points of weakly convex functions. Much like \cite{huang2021escaping}, our approach is based on~\cite{Davis2019ProximalMA}, but the resulting algorithms and their convergence guarantees are distinct from those in \cite{huang2021escaping}. We begin with a useful observation from~\cite{Davis2019ProximalMA}: near active strict saddle points $\bar x$, a certain $C^1$ smoothing, called the \emph{Moreau envelope}, is $C^2$ and has a strict saddle point at $\bar x$.
If one could \emph{exactly} execute the perturbed gradient method of~\cite{jin2019stochastic}, efficiency guarantees would then immediately follow.
While this is not possible in general, it is possible to \emph{inexactly} evaluate the gradient of the Moreau envelope by solving a strongly convex optimization problem.
Leveraging this idea, we extend the work~\cite{jin2019stochastic} to allow for inexact gradient evaluations, proving similar efficiency guarantees.

Setting the stage, we consider a minimization problem
\begin{align}\label{eq:main_prob}
\minimize_{x \in \RR^d}  f(x)
\end{align}
where $f \colon \RR^d \rightarrow \RR\cup\{+\infty\}$ is closed and $\rho$-weakly convex, meaning the mapping $x \mapsto f(x) + \frac{\rho}{2}\|x\|^2$ is convex. Although such functions are nonsmooth in general, they admit a global $C^1$ smoothing furnished by the Moreau envelope. For all $\mu < \rho^{-1}$, the \emph{Moreau envelope} and the \emph{proximal mapping} are defined to be
\begin{align}
  f_\mu (x) = \min_{y \in \RR^d}~ f(y) + \frac{1}{2\mu} \|y - x\|^2 \quad \text{and} \quad
  \prox_{\mu f}(x) = \argmin_{y \in \RR^d} f(y) + \frac{1}{2\mu} \|y - x\|^2, \label{eq:moreau}
\end{align}
respectively.
The minimizing properties of $f$ and $f_\mu$ are moreover closely aligned, for example, their first-order critical points and local/global minimizers coincide.
Inspired by this relationship, this work thus seeks $(\eps, \epsh)$-\emph{approximate second-order critical points} $x$ of $f_\mu$, satisfying:
\begin{equation}
\label{eq:28moreau}
\|\nabla f_\mu(x)\| \leq \eps \qquad \text{and} \qquad \lambda_{\min}(\nabla^2 f_\mu(x)) \geq -\epsh.
\end{equation}
An immediate difficulty is that $f_\mu$ is not $C^2$ in general. Indeed, the seminal work \cite{lemarechal1997practical} shows $f_\mu$ is $C^2$-smooth {\em globally}, if and only if, $f$ is $C^2$-smooth globally. Therefore assuming that  $f_\mu$ is $C^2$ globally is meaningless for nonsmooth optimization.
Nevertheless, known results in \cite{drusvyatskiy2016generic} imply that for ``generic'' semialgebraic functions, $f_\mu$ is locally $C^2$ near $x$ whenever $\|\nabla f_\mu (x)\|$ is sufficiently small.

Turning to algorithm design, a natural strategy is to apply a ``saddle escaping" gradient method~\cite{jin2019stochastic} directly to $ f_\mu$.
This strategy fails in general, since it is not possible to evaluate the gradient
 \begin{align}\label{eq:grad_moreau}
 \nabla f_\mu (x) = \frac{1}{\mu}(x - \prox_{\mu f}(x))
\end{align}
in closed form.
Somewhat expectedly, however, our \textbf{first contribution} is to show that one may extend the results of~\cite{jin2019stochastic} to allow for \emph{inexact} evaluations $\oracle(x) \approx \nabla f_\mu(x)$ satisfying
$$
\|\oracle(x) - \nabla f_\mu(x)\| \leq a\|\nabla f_\mu(x)\| + b \qquad \text{for all }x \in \RR^d,
$$
for appropriately small $a, b \geq 0$.
The algorithm (Algorithm~\ref{alg:pertInexGD}) returns a point $x$ satisfying~\eqref{eq:28moreau}, with $\tilde{\cO}(\max\{\eps^{-2}, \epsh^{-4}\})$ evaluations of $G$, matching the complexity of~\cite{jin2019stochastic}.

Our \textbf{second contribution} constructs approximate oracles $\oracle(x)$, tailored to common problem structures.
Each oracle satisfies
$$
\oracle(x) = \mu^{-1}\left(x - \aprox_{\mu f} (x)\right),
$$
where $\aprox_{\mu f}$ is an approximate minimizer of the \emph{strongly convex} subproblem defining $\prox_{\mu f}(x)$.
Since the subproblem is strongly convex, we construct $\aprox_{\mu f}$ from $K$ iterations of off-the-shelf first-order methods for convex optimization.
We focus in particular on the class of \emph{model-based methods}~\cite{davis2019stochastic}. Starting from initial point $x_0 = x$, these methods attempt to minimize $f(y) + \frac{1}{2\mu} \|y - x\|^2$ by iterating
\begin{align}\label{eq:mbasp}
x_{k+1} = \argmin_{y \in \RR^d} \left\{f_{x_k}(y) + \frac{1}{2\mu}\|y - x\|^2 + \frac{\theta_k}{2}\|y - x_{k}\|^2\right\},
\end{align}
where $\theta_k > 0$ is a control sequence and for all $z \in \RR^d$, the function $f_{z} \colon \RR^d \rightarrow \RR \cup\{+\infty\}$ is a local weakly convex model of $f$.
In Table~\ref{table:update_map}, we show three models, adapted to possible decompositions of $f$. In Table~\ref{table:overal_complexity}, we show how the model function $f_{z}$ influences the total complexity $\tilde{\cO}(K\times  \max\{\eps^{-2}, \epsh^{-4}\})$ of finding a second order stationary point of $f_\mu$~\eqref{eq:28moreau}. In short, prox-gradient and prox-linear methods require $\tilde{\cO}(\max\{\eps^{-2}, \epsh^{-4}\})$ iterations of~\eqref{eq:mbasp}, while prox-subgradient methods require $\tilde{\cO}(d \max\{\eps^{-6}\epsh^{-6}, \epsh^{-18}\})$. The efficiency of the prox-gradient method directly matches the analogous guarantees for the perturbed gradient method in the smooth setting \cite{jin2019stochastic}. The convergence guarantee of the prox-subgradient method has no direct analogue in the literature. Extensions for stochastic variants of these algorithms follow trivially, when the proximal subproblem \eqref{eq:mbasp} can be approximately  solved with high probability (e.g. using \cite{harvey2019tight,harvey2019simple,kakade2008generalization,rakhlin2011making}).
 The rates for the prox-gradient and prox-linear method are analogous to those in \cite{huang2021escaping}, which uses an algorithm-dependent measure of stationarity. Although the algorithms and the results in our paper and in \cite{huang2021escaping} are mostly of theoretical interest, they do suggest that efficiently escaping from saddle points is possible in nonsmooth optimization.

\begin{table}[t]
\begin{center}
	\begin{tabular}{|l |l |l|}
		\hline
		Algorithm& Objective  & Model function $f_z(y)$ \\ [0.5ex]
		\hline\hline
		Prox-Subgradient~\cite{davis2019stochastic}  & $\displaystyle  l(y) + r(y)$ & $\displaystyle l(z) + \dotp{v_z, y - z} + r(y)$ \\
		\hline
		Prox-gradient  & $\displaystyle F(y) + r(y) $  & $\displaystyle F(z) + \dotp{\nabla F(y), y - z} + r(y)$  \\
		\hline
		Prox-linear~\cite{fletcher1982model}  & $\displaystyle h(c(x))+r(x)$ & $\displaystyle  h(c(x)+\nabla c(x)(y-x))+r(y)$  \\
		\hline
	\end{tabular}
      \end{center}
\caption{The three algorithms with the update~\eqref{eq:mbasp}; we assume $h$ is convex and Lipschitz, $r$ is weakly convex and possibly infinite valued, both $F$ and $c$ are smooth, and $l$ is Lipschitz and weakly convex on $\dom r$ with $v_z \in \partial l(z)$.}\label{table:update_map}
\end{table}

\begin{table}[h!]
\begin{center}
	\begin{tabular}{|l |l|}
		\hline
		Algorithm to Evaluate $g(x)$   & Overall Algorithm Complexity \\ [0.5ex]
		\hline\hline
		Prox-Subgradient~\cite{davis2019stochastic}   & $\displaystyle \tilde{\cO}(d \max\{\eps^{-6}\epsh^{-6}, \epsh^{-18}\})$ \\
		\hline
		Prox-gradient    & $\displaystyle \tilde{\cO}(\max\{\eps^{-2}, \epsh^{-4}\})$  \\
		\hline
		Prox-linear~\cite{fletcher1982model}   & $\displaystyle  \tilde{\cO}(\max\{\eps^{-2}, \epsh^{-4}\})$  \\
		\hline
	\end{tabular}
      \end{center}
      \caption{The overall complexity of the proposed algorithm $\tilde{\cO}(K \times \max\{\eps^{-2}, \epsh^{-4}\})$, where $K$ is the number of steps of~\eqref{eq:mbasp} required to evaluate $g(x)$. The rate for Prox-subgradient holds in the regime $\eps = \cO(\epsh)$.}\label{table:overal_complexity}
\end{table}

\textbf{Related work.} We highlight several approaches for finding second-order critical points. Asymptotic guarantees have been developed in deterministic \cite{lee2016gradient, lee2019First, Davis2019ProximalMA} and stochastic settings \cite{pemantle1990nonconvergence}.
Other approaches explicitly leverage second order information about the objective function, such as full Hessian or Hessian vector products computations~\cite{nesterov2006cubic, Curtis_2016, agarwal2020adaptive,carmon2018accelerated, DBLP:conf/stoc/AgarwalZBHM17, royer2018complexity, Royer_2019, O_Neill_2020, curtis2021trust}.
Several methods exploit only first-order information combined with random perturbations~\cite{ge2015escaping, jin2019stochastic, DBLP:conf/icml/DaneshmandKLH18, DBLP:conf/colt/JinNJ18, jin2018local}.
The work \cite{jin2018local} also studies saddle avoiding methods with inexact gradient oracles $\oracle$; a key difference: the oracle of~\cite{jin2018local} is the gradient of a smooth function $\oracle =  \nabla  g$.
Several existing works have developed methods that find second-order stationary points of manifold \cite{criscitiello2019efficiently, sun2019escaping}, convex \cite{mokhtari2018escaping, nouiehed2018convergence,  DBLP:conf/nips/LuRYHH20, xie2021complexity}, and low-rank matrix constrained problems \cite{zhu2021global, ONeill2020ALD}.

\textbf{Road map.} In Section \ref{sec:prelim} we introduce the preliminaries. Section \ref{sec:escaping} presents a result for finding second-order stationary points with inexact gradient evaluations. Section \ref{sec:application} develops several oracle mappings that approximately evaluate the gradient of the Moreau Envelope and derives the complexity estimates of Table~\ref{table:overal_complexity}.

\section{Preliminaries}\label{sec:prelim}
This section summarizes the  notation that we use throughout the paper. We endow $\RR^d$ with the standard inner product $\langle x,y \rangle := x^\top y$ and the induced norm $\|x\|_2:=\sqrt{\langle x,x\rangle}$.  The closed unit ball in $\RR^d$ will be denoted by $\BB^d := \{x  \in  \RR^d \mid \|x\|\leq 1\}$, while a closed ball of radius $r>0$ around a point $x$ will be written as $\BB_{r}^d(x)$. When the dimension is clear from the context we write $\BB$. Given a function $\varphi\colon \RR^d \to \Rb$, the domain and the epigraphs of $\varphi$ are given by
 $ \dom \varphi = \{x \in \RR^d \mid \varphi(x) < \infty\} \text{ and } \epi \varphi = \{(x, r) \mid \varphi(x) \leq r\}.$ A function $\varphi$ is called closed if $\epi \varphi$ is a closed set. The distance of a point $x \in \RR^d$ to a set $\cM\subseteq \RR^d$ is denoted by
 $ \dist(x, \cM) = \inf_{y \in  \cM}\|x - y\|. $
The symbol $\|A\|$ denotes the operator norm of a matrix $A$, while
the maximal and minimal eigenvalues of a symmetric matrix $A$ will be denoted by $\lambda_{\max}(A)$ and $\lambda_{\min}(A)$, respectively.
For any bounded measurable set $Q \subset \RR^d$, we let $\Unif(Q)$ be the uniform distribution over $Q$.

We will require some basic constructions from Variational Analysis as described for example in the monographs \cite{RW98,Mord_1,cov_lift}. Consider a closed function $f\colon\RR^d \to\Rb$ and a point $x$, with $f(x)$ finite. The {\em subdifferential} of $f$ at $x$, denoted by $\partial f(x)$, is the set of all vectors $v\in\RR^d$ satisfying
\begin{equation}\label{eqn:subgrad_defn}
f(y)\geq f(x)+\langle v,y-x\rangle +o(\|y-x\|_2)\quad \textrm{as }y\to x.
\end{equation}
We set $\partial f(x) = \emptyset$ when $x \notin \dom f$. When $f$ is $C^1$ at $x \in \RR^d$, the subdifferential $\partial f(x)$ consists of the gradient $\{\nabla f(x)\}$. When $f$ is convex, it reduces to the subdifferential in the sense of convex analysis. In this work, we will primarily be interested in the class of $\rho$-weakly convex functions, meaning those for which $x \mapsto f(x) + \frac{\rho}{2}\|x\|^2$ is convex. For $\rho$-weakly convex functions the subdifferential satisfies:
$$
f(y) \geq f(x) + \dotp{v, y - x} - \frac{\rho}{2} \|y - x\|^2, \qquad \text{for all } x, y \in \RR^d, v \in \partial f(x).
$$
Finally, we mention that a point $x$ is a \emph{first-order critical point} of $f$ whenever the inclusion $0 \in \partial f(x)$ holds.

\section{Escaping saddle points with inexact gradients} \label{sec:escaping}
In this section, we analyze an inexact gradient method on smooth functions, focusing on convergence to second-order stationary points. The consequences for nonsmooth optimization, which will follow from a smoothing technique, will be explored in Section~\ref{sec:escaping}.

We begin with the following standard assumption, which asserts that the function $f$ in question has a globally Lipchitz continuous gradient.

\begin{assumption}[Globally Lipschitz gradient]\label{ass:localSmooth}
	Fix a function $\smooth\colon\RR^d\to\RR$ that is bounded from below and whose gradient is globally Lipschitz continuous with constant $\Lg$, meaning
    \[\|\nabla \smooth(x) - \nabla \smooth(y)\| \leq \Lg\|x-y\| \qquad  \text{for all }x,y\in \RR^d.\]
\end{assumption}

The next assumption  is more subtle: it requires the Hessian $\nabla^2 \smooth$ to be Lipschitz continuous on a neighborhood of any point where the gradient is sufficiently small. When we discuss consequences for nonsmooth optimization in the later sections, the fact that $f$ is assumed to be $C^2$-smooth only locally will be crucial to our analysis.

\begin{assumption}[Locally Lipschitz Hessian]\label{ass:localSmooth2}
	Fix a function $\smooth\colon\RR^d\to\RR$ and assume that there exist positive constants $\alphac, \betac,\Lh$ satisfying the following:
		For any point $\bar x$ with $\|\nabla \smooth(\bar x)\| \leq \alphac$, the function $\smooth$ is $C^2$-smooth on $\mathbb{B}_\betac(\bar x)$ and satisfies the Lipschitz condition:
		\[\|\nabla^2 \smooth(x) - \nabla^2 \smooth(y) \| \leq \Lh\|x-y\| \qquad \text{for all }x,y \in B_\beta(\bar x).\]
\end{assumption}

We aim to analyze an inexact gradient method for minimizing the function $f$ under Assumptions~\ref{ass:localSmooth} and \ref{ass:localSmooth2}. The type of inexactness we allow is summarized by the following oracle model.
\begin{definition}[Inexact oracle]\label{ass:inexact}
A map $\oracle\colon \RR^d \rightarrow \RR^d$ is an \emph{$(\cc, \ec)$-inexact gradient oracle} for $f$ if it satisfies
\begin{equation}
  \norm{\nabla \smooth(x) - \oracle(x)} \leq \cc \cdot \norm{\nabla \smooth(x)} + \ec \qquad \forall x \in \RR^d.
\end{equation}
\end{definition}

Turning to algorithm design, the method we introduce (Algorithm \ref{alg:pertInexGD}) directly extends the perturbed gradient method introduced in \cite{jin2019stochastic} to inexact gradient oracles in the sense of Definition~\ref{ass:inexact}. The convergence guarantees for the algorithm will be based on the following explicit setting of parameters. Fix target accuracies  $\eps, \epsh>0$ and choose any $\Delta_\smooth\geq \smooth(x_0)-\inf \smooth$. We first define the {\em auxiliary parameters}:
\begin{equation}\small
  \label{eq:defGamma}
  \phi := 2^{24} \max\left\{1, 5\frac{\Lh\eps }{\Lg\epsh}\right\} \frac{\Lg^2 }{ \delta}\sqrt{{d}} \left(\deltac \max\left\{\frac{\Lh^2}{\epsh^{5}}, \frac{1}{\eps^{2}\epsh^{1}}\right\}+\frac{1}{\epsh^2}\right)  \quad\text{and} \quad  \gammac := \log_2 ( \phi \log_2(\phi)^{8}),
\end{equation}
and
$$\Fc = \frac{1}{800 \gammac^3} \frac{1-\cc}{(1+\cc)^2} { \frac{\epsh^3}{\Lh^2}} \qquad \text{and} \qquad \Rc = \frac{1}{4\gammac} {\frac{\epsh}{\Lh}}.$$
The parameters required by the algorithm are then set as
\begin{align}\begin{split}\label{eq:para}
    &\eta = \frac{1-\cc}{(1+\cc)^2}  \frac{1}{\Lg} , \quad r = \frac{\epsh^2}{400\Lh\gamma^3} \min\left\{1, \frac{\Lg \epsh}{5\eps \Lh}\right\} , \quad \Kc = \frac{(1+\cc)^2}{(1-\cc)} \frac{\Lg}{\epsh} \gammac.
   \end{split}
\end{align}
 The following is the main result of the section. The proof follows closely the  argument in \cite{jin2019stochastic} and therefore appears in Appendix \ref{sec:main}.
\begin{algorithm2e}[t]
  \KwData{$x_0 \in \RR^d$, $T\in \mathbb{N}$, and $\eta, r,\eps,\Kc >0$}
  Set $\tp = -\Kc$\\
  {\bf Step}  $t=0,\ldots, T$\textbf{:}\\
  $ \qquad$Set $u_t = 0$\\
  $ \qquad${\bf If} $\norm{\oracle(x_t)} \leq \eps/2$ {\bf and} $t - \tp \geq \Kc$:\\
  $\qquad \qquad$Update $\tp = t$ \\
  $\qquad \qquad$Draw perturbation $u_t\sim \text{Unif} (r \mathbb{B})$ \\
  $\qquad$Set $\displaystyle x_{t+1} \leftarrow x_{t} -  \eta \cdot (\oracle(x_t)+u_t)$.
  \caption{Perturbed inexact gradient descent}
  \label{alg:pertInexGD}
\end{algorithm2e}

\begin{thm}[Perturbed inexact gradient descent]\label{thm:main} Suppose that $\smooth \colon\RR^d \rightarrow \RR$ is a function satisfying Assumptions~\ref{ass:localSmooth} and \ref{ass:localSmooth2} and $\oracle\colon \RR^d\rightarrow \RR^d$ is an $(\cc, \ec)$-inexact gradient oracle for $\smooth$. Let $\delta\in  (0, 1)$,  $\eps \in (0, \alpha)$,   $\epsh \in (0, \min\{4 \gamma \beta \Lh,\Lg,  \Lg^2\})$, and suppose that
  \begin{align*}
    \begin{split}
      &\cc \leq \min\left\{\frac{1}{20}, \frac{1}{\Lg \eta M 2^{\gamma  +2}}, \frac{R}{\eps \eta M 2^{\gamma  +2}}\right\}\quad \text{and}\\
      &\ec \leq \min\left\{{\frac{\eps}{64}}, \left(\frac{\Fc}{40 \eta M}\right)^{1/2}, \left(\frac{\Lg\Fc}{M(5\Lg+1)}\right)^{1/2}, \frac{R}{\Kc\eta 2^{(\gamma+2)}}\right\}.
    \end{split}
    \end{align*}
    Then with probability at least $1-\delta$, at least one iterate generated by Algorithm \ref{alg:pertInexGD} with parameters \eqref{eq:para} is a
    $(\eps, \epsh)$-second-order critical point of $\smooth$ after
  \begin{equation}
    T = 8\Delta_\smooth \max \left\{2\frac{\Kc}{\Fc}, \frac{256}{(1-\cc)\eta \eps^2}\right\} +4\Kc = \tilde{\mathcal{O}} \left(\Lg \Delta_\smooth \max\left\{ \frac{\Lh^2}{\epsh^4}, \frac{1}{\eps^2}\right\} \right) \quad \text{iterations.}
    \end{equation}
  \end{thm}
  The necessary bounds for $\cc$ and $\ec$ can be estimated as
    \begin{align}\begin{split}
    \label{eq:19}
    &    \cc \lesssim \frac{\delta}{ \Lg^3\deltac}\cdot d^{-1/2} \cdot \min\left\{\frac{\epsh^6}{{\Lh^2}}, \eps^2\epsh^2\right\} \cdot \min\left\{{1}, \frac{\Lg\epsh}{\Lh\eps}\right\}^2 \quad \text{and} \\
    &  \ec \lesssim\frac{\delta}{\Lg^2\Lh \deltac}\cdot d^{-1/2} \cdot {\min\left\{\frac{\epsh^{7}}{\Lh^2}, \eps^2\epsh^3\right\}} \cdot \min\left\{{1}, \frac{\Lg\epsh}{\Lh\eps}\right\},
    \end{split}
    \end{align}
    where the symbol ``$\lesssim$'' denotes inequality up to polylogarithmic factors.
    Thus, Algorithm~\ref{alg:pertInexGD} is guaranteed to find a second order stationary point efficiently, provided that the gradient oracles are highly accurate. In particular, when $a = b = 0$ we recover the known rates from \cite{jin2019stochastic}.

\section{Escaping saddle points of the Moreau envelope} \label{sec:application}

In this section, we apply Algorithm~\ref{alg:pertInexGD} to the Moreau Envelope~\eqref{eq:moreau} of the weakly convex optimization problem~\eqref{eq:main_prob} in order to find a second order stationary point of $f_\mu$~\eqref{eq:28moreau}. We will see that a variety of standard algorithms for nonsmooth convex optimization can be used as inexact gradient oracles for the Moreau envelope. Before developing those algorithms, we summarize our main assumptions on $f_\mu$, describe why approximate second order stationary points of $f_\mu$ are meaningful for $f$, and show that Assumption~\ref{ass:localSmooth2}, while not automatic for general $f_\mu$, holds for a large class of semialgebraic functions.

As stated in the introduction, for $\mu < \rho^{-1}$, the Moreau envelope is an everywhere $C^1$ smooth with  Lipschitz continuous gradient. In particular,
\begin{quote}
Assumption~\ref{ass:localSmooth} holds automatically for $f_\mu$ with $\Lg = \max\left\{\mu^{-1}, \frac{\rho}{1-\mu \rho}\right\}$.
\end{quote}
See for example \cite{Davis2019ProximalMA} for a short proof.
Assumption~\ref{ass:localSmooth2}, however, is not automatic, so we impose the following assumption throughout.

\begin{assumption}\label{ass:LipHessian}
  Let $f\colon \RR^d \rightarrow \RR \cup \{\infty\}$ be a closed $\rho$-weakly convex function whose Moreau envelope $f_{\mu}$ satisfies Assumption~\ref{ass:localSmooth2} with constants $\alphac, \betac,\Lh$.
\end{assumption}

Turning to stationarity conditions, a natural question is whether the second order condition~\eqref{eq:28moreau} is meaningful for $f$.
The next proposition shows that the condition~\eqref{eq:28moreau} implies the existence of an approximate quadratic minorant of $f$ with small slope and curvature at a nearby point.

\begin{figure}
	\centering
	\begin{subfigure}[b]{0.65\textwidth}
		\centering
		\includegraphics[width=\textwidth]{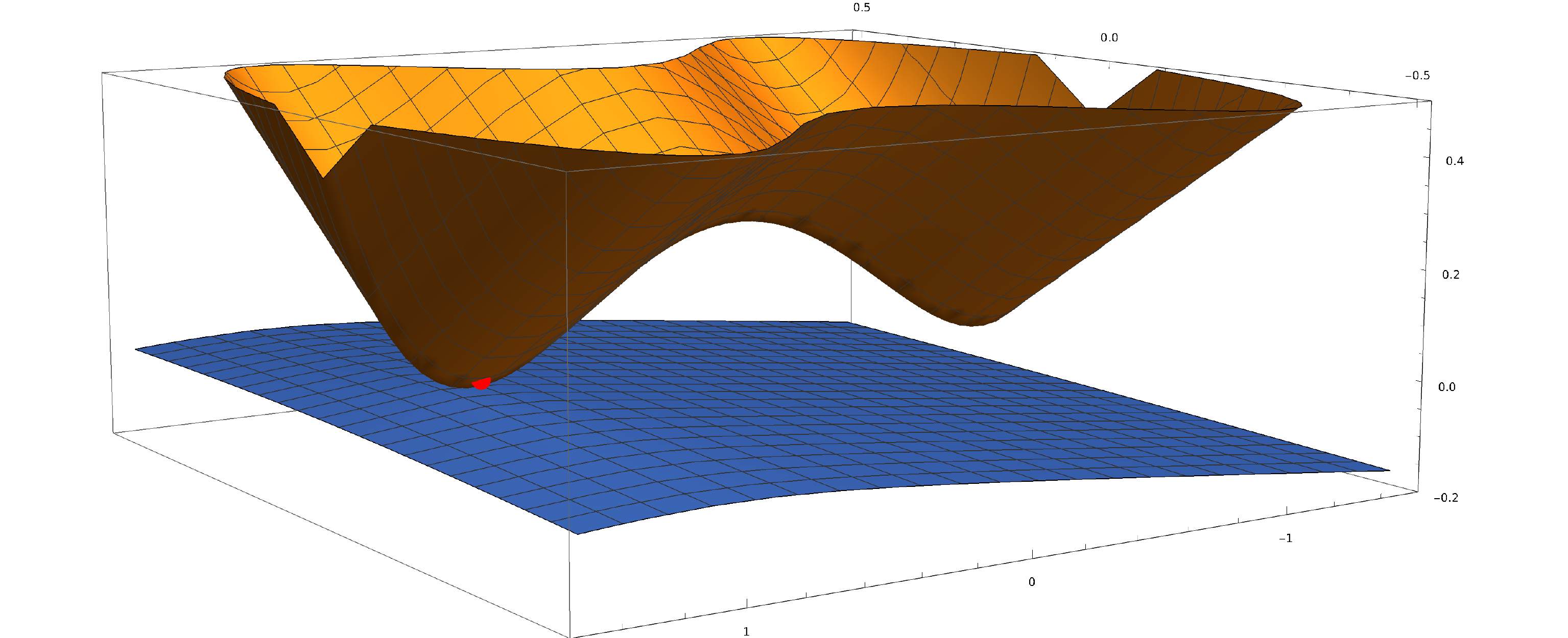}
	\end{subfigure}
	\hfill
	\begin{subfigure}[b]{0.3\textwidth}
		\centering
		\includegraphics[width=\textwidth]{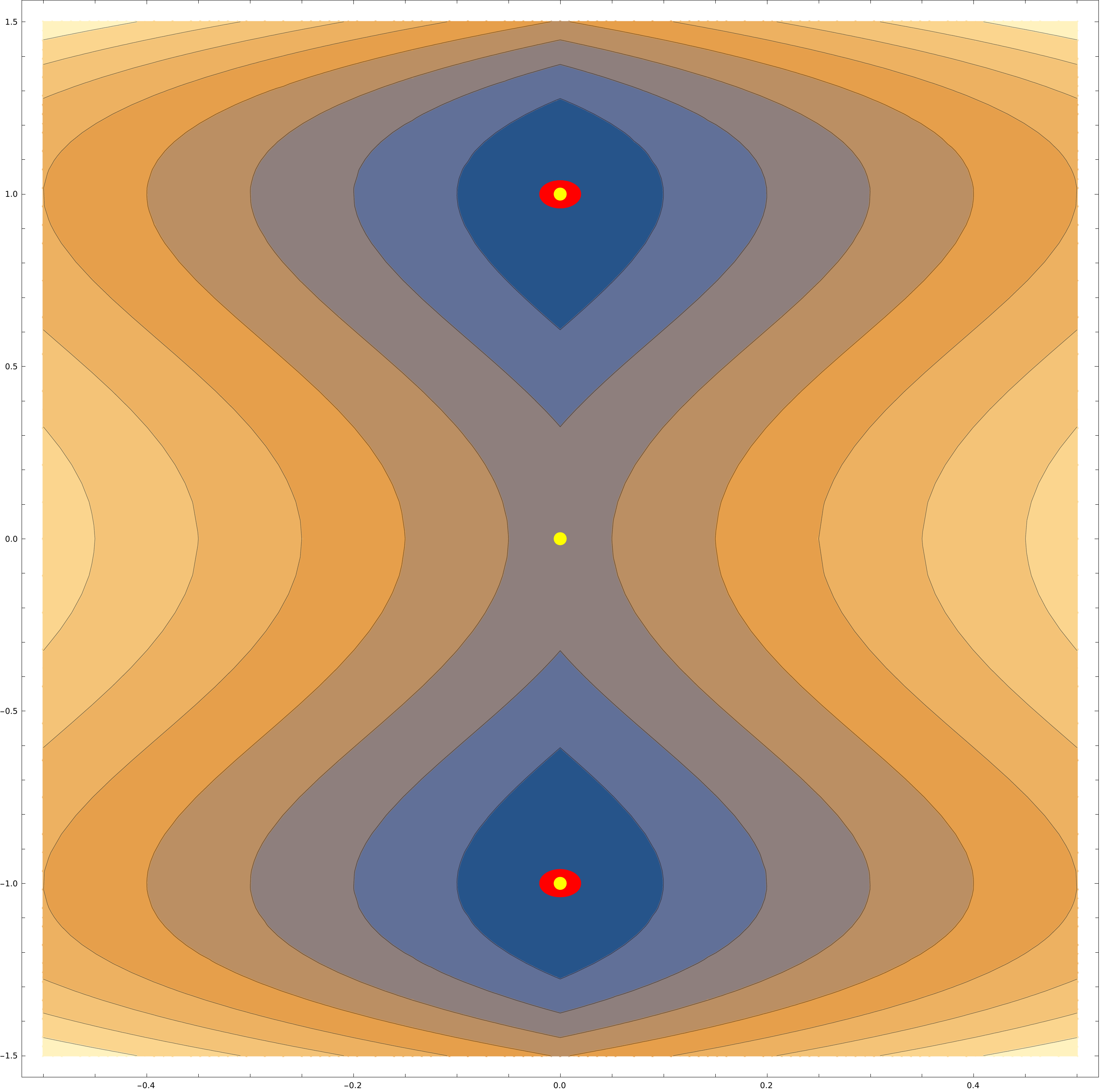}
	\end{subfigure}
	\caption{Critical poins of $f$ in \eqref{eq:34}. We use $\eps=\epsh= 0.04.$ On the left: The function, a point $(x, f(x))$ with $x$ an $(\eps, \epsh)$-second-order critical point of $f_\mu$ and its corresponding quadratic $q(\cdot)$. On the right: The set of first-order critical points of $f$ (yellow) and the set of $(\eps, \epsh)$-second-order critical points of $f_\mu$ (red).}
	\label{fig:1}
      \end{figure}

We defer the proof to Appendix \ref{sec:proofQuadratic}.

\begin{proposition}\label{prop:quadratic}
  Consider $f\colon\RR^d \rightarrow \Rb$ satisfying Assumption~\ref{ass:LipHessian}.  Assume that $x \in \RR^d$ is an $(\eps, \epsh)$-second order critical point of $f_\mu$  with $\eps\leq \min\{\alpha, \frac{\epsh}{2\Lh \mu}\}$ and let $\hat x:= \prox_{\mu f}(x)$.
 Then there exists a quadratic function $q\colon\RR^d \rightarrow \RR$ and a neighborhood $\cU = B_{{3\epsh}/{4\Lh}}(\widehat x)$ of $x$ for which the following hold.
  \begin{enumerate}
   \item \label{part:1} (\textbf{Nearby point}) The point $\widehat x$ is close to $x$: $\|x - \widehat x\| \leq \mu \cdot \eps$.
  \item \label{part:2} (\textbf{Minorant}) For any $y \in \cU$, we have $q(y) \leq f(y).$
  \item \label{part:3} (\textbf{Small subgradient}) The quadratic has a small gradient at $\widehat x$:
$$ \|\nabla q(\widehat x)\| \leq \eps.$$
  \item \label{part:4} (\textbf{Small negative curvature}) The quadratic has small negative curvature:
  $$\nabla^2 q(\widehat x) \succeq -{3\epsh}.$$
  \item \label{part:5} (\textbf{Approximate match}) The quadratic almost matches the function at $\widehat x$: $$  f(\widehat x)  - q(\widehat x) \leq \frac{{\mu}}{2}\left(1+ 3 \mu \epsh\right)\eps^2.$$
  \end{enumerate}
\end{proposition}

In Figure \ref{fig:1}, we illustrate the proposition with the following nonsmooth function:
\begin{equation}
  \label{eq:34}
  f(x, y) = |x| + \frac{1}{4}(y^2-1)^2.
\end{equation}
The Moreau envelope of this function has three first-order critical points: a strict saddle point $(0, 0)$ and two global minima $(-1, 0)$, and $(1, 0)$. As shown in  the right plot of Figure~\ref{fig:1}, approximate second-order critical points of $f_\mu$  cluster around minimizers of $f$. In addition, the left plot of Figure~\ref{fig:1} shows the lower bounding quadratic from Proposition~\ref{prop:quadratic}.

Finally, we complete this section by showing that Assumption~\ref{ass:LipHessian} is reasonable: it holds for generic semialgebraic functions.\footnote{A set is semialgebraic if its graph can be written as a finite union of sets each defined by finitely many polynomial inequalities.}

\begin{thm}\label{thm:generic}
Let $f\colon\RR^d \rightarrow \Rb$ be a semi-algebraic $\rho$-weakly-convex function. Then, the set of vectors $v\in \R^d$ for which the tilted function $g(x;v) = f(x) + \dotp{v,x}$ satisfies Assumption~\ref{ass:LipHessian} has full Lebesgue measure.
\end{thm}

The proof appears in Appendix \ref{sec:proofGeneric}, and is a small modification of the argument in \cite{Davis2019ProximalMA}.

\subsection{Inexact Oracles for the Moreau Envelope}
In this section, we develop inexact gradient oracles for $\nabla f_\mu = \mu^{-1}(x - \prox_{\mu f}(x))$. Leveraging this expression, our oracles will satisfy
\begin{equation}
  \label{eq:32}
 \oracle(x) = \mu^{-1}\left(x - \aprox_{\mu f} (x)\right),
\end{equation}
where $\aprox_{\mu f}$ is the output of a numerical scheme that solves \eqref{eq:moreau}. To ensure $\oracle$ meets the conditions of Definition~\ref{ass:inexact}, we require that
 $$ \| \aprox_{\mu f} (x) -\prox_{\mu f}(x)\| \leq \cc \cdot \|x - \prox_{\mu f}(x)\| + \mu \cdot \ec.$$
for some constants $\cc \in  (0, 1)$ and $\ec > 0$.

Since $f$ is $\rho$-weakly convex, evaluating $\prox_{\mu f}(x_k)$ amounts to minimizing the $(\mu^{-1} - \rho)$-strongly convex function $f(x) + \frac{1}{2\mu}\|x-x_k\|^2$.
We now use this strong convexity to derive efficient proximal oracles via a class of algorithms called \emph{model-based methods}~\cite{davis2019stochastic}, which we now briefly summarize.  Given a minimization problem $ \min_{x \in \RR^d} g(x)$, where $g$ is strongly convex,
a \emph{model-based method} is an algorithm that recursively updates
\begin{equation}
  \label{eq:17}
  x_{k+1} \leftarrow \argmin_x~ g_{x_k}(x) + \frac{\theta_t}{2} \|x - x_k\|^2,
\end{equation}
where $g_{x_k}\colon\RR^d \rightarrow \Rb$ is a function that approximates $g$ near $x_k$.
Returning to the proximal  subproblem, say we wish to compute $\prox_{\mu f}(x_0)$ for some given $x_0$. We consider an inner loop update of the form
\begin{equation}
  \label{eq:27}
x_{k+1} \leftarrow \argmin_{x\in \RR^d}~f_{x_k}(x) + \frac{1}{2\mu} \|x-x_0\|^2 + \frac{\theta_k}{2}\|x-x_k\|^2,
\end{equation}
where $f_{x_k}\colon\RR^d \rightarrow \Rb$ is a function that locally approximates $f$ (see Table~\ref{table:update_map} for three examples). Completing the square, this update can be equivalently written as a proximal step on $f_{x_k}$, where the reference point is a weighted average of $x_0$ and $x_k$ as summarized in Algorithm \ref{alg:ModelBasedOracle}. Turning to complexity, we note that the approximation quality of a model governs the speed at which iteration \eqref{eq:27} converges. In what follows, we will present two families of models with different approximation properties, namely one- and two-sided models. We will see that models with double-sided accuracy require fewer iterations to approximate $\prox_{\mu f}(x_0)$.

\begin{algorithm2e}[t]
  \KwData{Initial point $x_0 \in \RR^d$.}
  {\bf Parameters:} Stepsize $\theta_k > 0$, Flag \texttt{one\_sided}.\\
  {\bf Output:} Approximation of $\prox_{\mu f}(x_0)$.\\
  {\bf Step $k$} ($k \leq \Ki+1$)\textbf{:}\\
  $\qquad x_{k+1} \leftarrow \argmin_{x\in \RR^d} f_{x_k}(x) + \frac{1+\theta_k \mu}{2\mu}\left\|x - \frac{\left(x_0 + \theta_k\mu \cdot x_k\right)}{1+\theta_k\mu}\right\|^2$\\
  {\bf If }\texttt{one\_sided} \textbf{:} \\
  $\qquad \bar x_\Ki= \frac{2}{(\Ki+2)(\Ki+3) - 2} \sum_{k=1}^{\Ki+1}(t+1) x_k$\\
  $\qquad \textbf{return } \bar x_\Ki$\\
  {\bf Else:}\\
  $\qquad \textbf{return } x_\Ki$
  \caption{\textsc{ProxOracle}$_{\mu f}^\Ki$}
  \label{alg:ModelBasedOracle}
\end{algorithm2e}

\subsubsection{One-sided models}

We start by studying models that globally lower bound the function and agree with it at the reference point. Subgradient-type models are the canonical examples, and we will discuss them shortly.

\begin{assumption}[One-sided model]\label{ass:oneSided}  Let $f = l + r$, where $r \colon\RR^d \rightarrow \RR \cup\{+\infty\}$ is a closed function and $l \colon\RR^d \rightarrow \RR$ is locally Lipschitz.  Assume there exists  $\tau > 0$ and a family of models $l_x \colon \RR^d \rightarrow \RR$, defined for each $x \in \RR^d$, such that the following hold: For all $x \in \RR^d$, $l_x$ is $L$-Lipschitz on $\dom r$ and satisfies
\begin{equation}
l_x(x) = l(x) \qquad \text{and} \qquad l_x(y) - l(y) \leq \tau \|y - x\|^2  \qquad \text{for all }y \in \RR^d.
\end{equation}
In addition, for all $x \in \RR^d$, the model
$$
f_x := l_x + r
$$
is $\rho$-weakly convex.
\end{assumption}

Now we bound the number of iterations that are needed for Algorithm~\ref{alg:ModelBasedOracle} to obtain a $(\cc, \ec)$-inexact proximal point oracle with one-sided models. The algorithm outputs an average of the iterates with nonuniform weights that improves the convergence speed.
\begin{thm}\label{thm:oneSidedOracle}
  Fix $a, b > 0$ and let $f\colon \RR^d  \rightarrow\Rb$ be a $\rho$-weakly-convex function and let $f_x \colon \RR^d \rightarrow \Rb$ be a family of models that satisfy Assumption \ref{ass:oneSided} for $\tau = 0$. Let $\mu^{-1} > \rho$ be a constant, and set $\theta_k = \frac{(\mu^{-1}-\rho)}{2}(k+1)$ then Algorithm \ref{alg:ModelBasedOracle} with flag \emph{$\texttt{one\_sided} = \texttt{true}$} outputs an a point $\bar x_K$ such that
  $$\|\bar x_K - \prox_{\mu f}(x_0)\|_2 \leq \cc \cdot \|x_0 - \prox_{\mu f}(x_0)\|_2 + \mu \cdot \ec,$$
  provided the number of iterations is at least $\Ki \geq \frac{4}{\cc} + \frac{16L^2}{(1-\mu\rho)^2\ec^2}.$
\end{thm}

The proof of this result follows easily from Theorem 4.5 in \cite{davis2019stochastic} and thus, we omit it. By exploiting this rate, we derive a complexity guarantee with one-sided models.

\begin{thm}[\textbf{One-sided model-based method}]\label{thm:RateOneSided}
Consider an $\Lf$-Lipschitz $\rho$-weakly-convex function $f\colon\RR^d \rightarrow \Rb$ that satisfies Assumption \ref{ass:LipHessian} and a family of models $f_x$ satisfying Assumption \ref{ass:oneSided}. Then, for all sufficiently small $\eps > 0$, and any $\epsh >0$, $\delta\in (0, 1)$ there exists a parameter configuration $(\eta, r, \Kc)$ that ensures that with probability at least $1-\delta$ one of the first $T$ iterates generated by Algorithm~\ref{alg:pertInexGD} with gradient oracle
$$ g(x) = \mu^{-1}\left(x - \textsc{ProxOracle}_{\mu f}^K(x)\right)  \qquad \text (\text{Algorithm \ref{alg:ModelBasedOracle}})$$
is an $(\eps, \epsh)$-second-order critical point of $f_\mu$ provided that the inner and outer iterations satisfy
\begin{align}
 \begin{split} \label{eq:25}
& K = \tilde{\cO}\left((1-\mu \rho)^{-2}{\Lf^2 \Lg^4\Lh^2 \deltanc^2} \cdot \frac{d}{\delta} \cdot {\max\left\{\frac{\Lh^{4}}{ \epsh^{14}}, \frac{1}{\eps^{4}\epsh^{6}}\right\}}\cdot\max\left\{ \frac{\Lh^2\eps^2}{\Lg^2\epsh^2}, 1\right\}\right) \;\; \text{and} \\ & T = \tilde{\cO}\left(\Lg\deltanc\max\left\{\frac{\Lh^2}{\epsh^{4}},\frac{1}{\eps^{2}}\right\}\right)
\end{split}
\end{align}
  where $\Lg := \max\left\{\frac{1}{\mu}, \frac{\rho}{1-\mu \rho} \right\}$ and $\deltanc = f(x_0) - \inf f$.
\end{thm}

\begin{proof}
  This result is a corollary of Theorem \ref{thm:oneSidedOracle} and Theorem \ref{thm:main}. By \cite[Lemma 2.5]{Davis2019ProximalMA} and Assumption~\ref{ass:LipHessian} we conclude that the Moreau envelope satisfies the hypothesis of Theorem~\ref{thm:main}. Hence, the result follows from this theorem provided that we show that the gradient oracle is accurate enough.
  By Theorem \ref{thm:oneSidedOracle} if we set the number of iterations according to \eqref{eq:25} we get an inexact oracle that matches the assumptions of Theorem \ref{thm:main}
\end{proof}

The rate from Table~\ref{table:overal_complexity} follows by noting that $\max\left\{ \frac{\Lh^2\eps^2}{\Lg^2\epsh^2}, 1\right\} = 1$ when $\eps\leq \frac{\Lg}{\Lh}\epsh$.

\textbf{Example: proximal subgradient method.} Consider the setting of Assumption~\ref{ass:oneSided}, where $f = l + r$. Assuming that $l$ is $\tau$-weakly convex, it possesses an affine model:
$$
l_x(y) = l(x) + \dotp{v, y - x}, \qquad \text{ where } v \in \partial l(x).
$$
By weak convexity, $f_x = l_x + r$ satisfies Assumption~\ref{ass:oneSided}.
Moreover, the resulting update~\eqref{eq:27} reduces to the following proximal subgradient method:
$$
x_{k+1} = \prox_{\frac{\mu}{1+\theta_k\mu}r}\left(\frac{1}{1+\theta_k\mu}\left( x_0 + \theta_k \mu \cdot x_k - \mu\cdot v\right) \right).
$$
Theorem~\ref{thm:RateOneSided} applied to this setting thus implies the rate in Table~\ref{table:overal_complexity}.

\subsubsection{Two-sided models}

The slow convergence of one-sided model-based algorithms motivates stronger approximation assumptions. In this section we study models that satisfy the following assumption.
\begin{assumption}[Two-sided model]\label{ass:doubleSided} Assume that for any $x \in \RR^d$, the function $f_x\colon\RR^d \rightarrow \Rb$ is $\rho$-weakly convex and satisfies
\begin{equation}
  \label{eq:16}
|  f_x(y) - f(y) | \leq \frac{q}{2} \|y-x\|^2 \qquad \text{for all }y \in \RR^d.
\end{equation}
\end{assumption}
When equipped with double-sided models, model-based algorithms for the proximal subproblem converge linearly.

\begin{thm}\label{thm:twoSidedOracle}
  Suppose that $f\colon\RR^d  \rightarrow\Rb$ is a $\rho$-weakly-convex function, let $f_x$ be a family models satisfying Assumption \ref{ass:doubleSided}. Fix an accuracy level $\cc$. Set $\mu^{-1} > \rho + q$ and the stepsizes to $\theta_t = \theta > q$, then Algorithm \ref{alg:ModelBasedOracle} with flag \emph{$\texttt{one\_sided} = \texttt{false}$} outputs a point $x_K$ such that
  $$\|x_K - \prox_{\mu f}(x_0)\|_2 \leq \cc \cdot \|x_0 - \prox_{\mu f}(x_0)\|_2,$$
  provided that $K \geq 2\log (\cc^{-1})\log \left(\frac{\mu^{-1} - \rho + \theta}{q + \theta} \right)^{-1}.$
\end{thm}

We defer the proof of this result to Appendix \ref{sec:proofTwoSided}. Given this guarantee for two-sided models, we derive the following theorem. The proof is analogous to that of Theorem \ref{thm:RateOneSided}: the only difference is that we use Theorem \ref{thm:twoSidedOracle} instead of Theorem \ref{thm:oneSidedOracle}. Thus we omit the proof.

\begin{thm}[\textbf{Two-sided model-based method}]\label{thm:RateTwoSided}
Consider a  weakly convex function $f\colon\RR^d \rightarrow \Rb$ that satisfies Assumption \ref{ass:LipHessian} and a family of models $f_x$ satisfying Assumption \ref{ass:doubleSided}. Then for any $\delta\in (0, 1)$ and sufficiently small $\eps > 0$, there exists a parameter configuration $(\eta, r, \Kc)$ such that with probability at least $1-\delta$ one of the first $T$ iterates generated by Algorithm~\ref{alg:pertInexGD} with inexact oracle
$$ g(x) = \mu^{-1}\left(x - \textsc{ProxOracle}_{\mu f}^K(x)\right) \qquad \text (\text{Algorithm \ref{alg:ModelBasedOracle}})$$
is an $(\eps, \epsh)$-second-order critical point of $f_\mu$  provided that the inner and outer iterations satisfy
  $$K = \tilde{\cO}(1) \quad \text{and} \quad T =  \tilde{\cO}\left(\max\left\{\frac{1}{\mu}, \frac{\rho}{1-\mu \rho} \right\}(f(x_0)- \inf f)\min\left\{\Lh^2\eps^{-4}, \eps^{-2}\right\}\right).$$
\end{thm}

We close the paper with two examples of two-sided models.

\textbf{Example: Prox-gradient method}. Suppose that
$$
f = F + r
$$
where $r \colon \RR^d \rightarrow \RR\cup\{+\infty\}$ is closed and $\rho$-weakly convex and $F$ is $C^1$ with $q$-Lipschitz continuous derivative on $\dom r$. Then due to the classical inequality $$|F(y) - F(x) - \dotp{\nabla F(x),y-x}| \leq \frac{q}{2}\|y - x\|^2 \qquad\text{for all } x, y \in \dom r,$$ the model
$$
f_x(y) = F(x) + \dotp{\nabla F(x), y - x} + r(x),
$$
satisfies Assumption~\ref{ass:doubleSided}. Moreover, the resulting update~\eqref{eq:27} reduces to the following proximal gradient method:
$$
x_{k+1} = \prox_{\frac{\mu}{1+\theta_k\mu}r}\left(\frac{1}{1+\theta_k\mu}\left( x_0 + \theta_k \mu \cdot x_k - \mu\cdot \nabla F(x_k)\right)\right).
$$
Theorem~\ref{thm:RateTwoSided} applied to this setting thus implies the rate in Table~\ref{table:overal_complexity}.

\textbf{Example: Prox-linear method}. Suppose that
$$
f = h\circ c + r
$$
where $r \colon \RR^d \rightarrow \RR\cup\{+\infty\}$ is closed and $\rho$-weakly convex,  $h$ is $L$-Lipschitz and convex on $\dom r$, and $c$ is $C^1$ with $\beta$-Lipschitz Jacobian on $\dom r$. Then due to the classical inequality $\|c(y) - c(x) - \nabla c(x),y-x\| \leq \frac{\beta}{2}\|y - x\|^2$, we have
$$
|h(c(y)) - h(c(x) + \nabla c(x)(y-x))| \leq \frac{\beta L}{2}\|x - y\|^2, \qquad \text{for all } x, y \in \dom r.
$$
Consequently, the model
$$
f_x(y) = h(c(x) + \nabla c(x)(y-x)) + r(x),
$$
satisfies Assumption~\ref{ass:doubleSided} with $q = \beta L$. Moreover, the resulting update~\eqref{eq:27} reduces to the following prox-linear method~\cite{fletcher1982model}:
$$
x_{k+1} = \argmin_{y \in \RR^d} h(c(x_k) + \nabla c(x_k)(y-x_k)) + r(x) + \frac{1+\theta_k \mu}{2\mu} \left\| x - \frac{x_0 + \theta_k \mu \cdot x_k}{1+\theta_k \mu}\right\|^2.
$$
Theorem~\ref{thm:RateTwoSided} applied to this setting thus implies the rate in Table~\ref{table:overal_complexity}.

\bibliographystyle{plain}
\bibliography{bibliography}
\appendix

\section{Proof of Theorem \ref{thm:main}}\label{sec:main}
Throughout this section, we assume the setting of Theorem~\ref{thm:main}.

We begin by recording some inequalities that we will use later on.

\begin{lemma}
  \label{lemma:horrible}
  The following inequalities hold.
  \begin{enumerate}
  \item(\textbf{Radius}) $$\sqrt{32\eta \frac{(1+\cc)^2}{(1-\cc)} \Kc \Fc} + \eta r  < R.$$
  \item(\textbf{Function value}) $$\eps \eta r + \Lg\eta^2r^2/2 \leq F/2.$$
  \item(\textbf{Probability}) $$p :=\frac{T \Lg \frac{(1+\cc)^2}{(1-\cc)} \frac{\sqrt{d}}{ \epsh} \gamma^2 \max\left\{1, 5\frac{\Lh \eps}{\Lg \epsh}\right\} 2^{9}}{2^\gamma} \leq \delta.$$
  \end{enumerate}
\end{lemma}
\begin{proof}
We start with the first inequality, observe that
  $$
32 \eta  \frac{(1+a)^2}{(1-a)} \leq  32\frac{1}{\Lg} \quad \text{ and } \quad  \Fc \Kc = \frac{1}{800 \gammac^3} \frac{1-\cc}{(1+\cc)^2} { \frac{\epsh^3}{\Lh^2}} \cdot \frac{(1+\cc)^2}{(1-\cc)} \frac{\Lg}{\epsh} \gammac = \frac{\epsh^2\Lg }{800\Lh^2\gamma^2}.
$$
Therefore, since
\begin{align*}
\eta \leq\frac{1}{\Lg} \qquad \text{ and } \qquad  r = \frac{\epsh^2}{400\Lh\gamma^3} \min\left\{1, \frac{\Lg \epsh}{5\eps \Lh}\right\} \leq \frac{\epsh^2}{400\Lh\gamma^3},
\end{align*}
we have
  \begin{align*}
      \sqrt{32\eta \frac{(1+\cc)^2}{(1-\cc)} \Kc \Fc} + \eta r
                            &\leq \frac{1}{5\gamma} \frac{{\epsh}}{ \Lh} + \frac{\epsh^2}{400\Lg \Lh\gamma} \\
                                                          &\leq\frac{1}{5\gamma} {\frac{\epsh}{\Lh}}+ \frac{1}{400 \gamma}{\frac{\epsh}{\Lh}} < \frac{1}{4\gammac} {\frac{\epsh}{\Lh}} = R.
  \end{align*}
  where the third inequality follows from $\Lg/{ \epsh}\geq 1$.

 Now, we prove the second statement: $\eps \eta r + \Lg\eta^2r^2/2 \leq F/2$.
   Indeed, first recall the definition of $r$ above and that $\eta = \frac{1-\cc}{(1+\cc)^2}  \frac{1}{\Lg}$,
  $
  \Fc = \frac{1}{800 \gammac^3} \frac{1-\cc}{(1+\cc)^2} { \frac{\epsh^3}{\Lh^2}}.$ Thus, we bound the first term:
  $$
  \eps \cdot \eta  \cdot   r \leq \eps \cdot  \frac{1-\cc}{(1+\cc)^2}  \frac{1}{\Lg} \cdot  \frac{\epsh^2}{400\Lh\gamma^3}  \frac{\Lg \epsh}{5\eps \Lh} \leq  \frac{1-\cc}{(1+\cc)^2}\frac{\epsh^3}{2000\Lh^2\gamma^3} \leq \frac{2}{5}\Fc.
  $$
  Next, we bound the second term:
 \begin{align*}
  \frac{\Lg\cdot \eta^2 \cdot r^2}{2} &= \frac{1}{2} \Lg \cdot \left(\frac{1-\cc}{(1+\cc)^2}  \frac{1}{\Lg}\right)^2 \cdot \left(\frac{\epsh^2}{400\Lh\gamma^3}\right)^2 \\
  &= \frac{\epsh}{\Lg} \frac{1-a}{(1+a)^2} \frac{1}{400\gamma^3}\frac{1}{800\gamma^3}\frac{1-a}{(1+a)^2}\frac{\epsh^3}{\Lh^2} \\
  &\leq \frac{1}{400\gamma^3}\frac{1}{800\gamma^3}\frac{1-a}{(1+a)^2}\frac{\epsh^3}{\Lh^2} \leq \frac{\Fc}{10}
\end{align*}
where we used $(1- \cc)/(1+ \cc)^2 \leq 1$, $\epsh\leq \Lg$ and the simple inequality $1/400\gamma^3 \leq 1/10$.

Finally, we show that $p \leq \delta$. Recall that by definition,
  \begin{equation*} \small T = 8 \Delta_g \max \left\{\frac{\Kc}{\Fc} , \frac{256}{\eta \eps^2}\right\} + 4\Kc.
  \end{equation*}
We upper bound $T$ using $\Fc = \frac{1}{800 \gammac^3} \frac{1-\cc}{(1+\cc)^2} { \frac{\epsh^3}{\Lh^2}}$, $\Kc = \frac{(1+\cc)^2}{(1-\cc)} \frac{\Lg}{\epsh} \gammac,$ and $\eta =  \frac{1-\cc}{(1+\cc)^2}  \frac{1}{\Lg}$:
 \begin{align*}
 T &= 2^4 \frac{(1+a)^2}{1-a} \deltac \Lg \max\left\{800 \gamma^4 \frac{(1+a)^2}{1-a}\frac{\Lh^2}{\epsh^{4}},  \frac{256}{\eps^2}\right\} + 4 \frac{(1+\cc)^2}{(1-\cc)} \frac{\Lg}{\epsh} \gammac\\
 &\leq 2^4 \cdot 800 \left(\frac{(1+a)^2}{1-a}\right)^2 \cdot \Lg\gamma^4 \cdot \left(\deltac \max\left\{ \frac{\Lh^2}{\epsh^4}, \frac{1}{\eps^2} \right\} +  \frac{1}{\epsh}\right).
\end{align*}
This yields:
$$
p \leq \frac{ 2^{13} \cdot 800 \left(\frac{(1+a)^2}{1-a}\right)^3 \cdot \Lg^2\gamma^6 \sqrt{d}\cdot  \max\left\{1, 5\frac{\Lh \eps}{\Lg \epsh}\right\} \left(\deltac \max\left\{ \frac{\Lh^2}{\epsh^5}, \frac{1}{\eps^2\epsh} \right\} +  \frac{1}{\epsh^2}\right)}{2^\gamma}.
$$
Next, recall that $2^{\gamma} = \phi \cdot \log_2(\phi)^8$, where
 \begin{equation*}\small
\phi := 2^{24}  \frac{\Lg^2 }{ \delta}\sqrt{{d}} \max\left\{1, 5\frac{\Lh\eps }{\Lg\epsh}\right\} \left(\deltac \max\left\{\frac{\Lh^2}{\epsh^{5}}, \frac{1}{\eps^{2}\epsh}\right\}+\frac{1}{\epsh^2}\right).
\end{equation*}
Note that $\phi \geq 2^{24} \frac{\Lg^2}{\epsh^2} \geq 2^{24}$ since $\epsh\leq \Lg$.
Therefore,
\begin{align*}
p &\leq 2^{13} \cdot 800 \left(\frac{(1+a)^2}{1-a}\right)^3\frac{ \gamma^6}{2^{24}\log_2^8(\phi)}\delta \leq \delta
\end{align*}
where the final inequality follows from $\log_2(x \log_2(x)^8)^6 \leq \log_2\left(x\right)^8$ for any $x \geq 2^{24}$ and $2^{13} \times 800 \times \left(\frac{(1+a)^2}{1-a}\right)^3 \leq 2^{24}$ since $a\leq 1/20$.
\end{proof}

We assume that $\oracle$ is an $(\cc, \ec)$-inexact gradient oracle for $\smooth$. We derive two simple consequences of Definition~\ref{ass:inexact}.

\begin{lemma}\label{lemma:facts}
 Then we have that for any $x \in \RR^d$ the following inequalities hold:
  \begin{enumerate}
  \item \textbf{(Norm similarity)} $|\|\oracle(x)\| - \|\nabla \smooth(x)\| |\leq  a\|\nabla \smooth(x)\| + \ec.$
  \item \textbf{(Correlation)} $\hspace{23pt} \dotp{\nabla \smooth(x), \oracle(x)} \geq  (7/8)(1-a)\|\nabla \smooth(x)\|^2 - 2b^2.$
  \end{enumerate}
\end{lemma}
\begin{proof}
Throughout the proof we let $v = \nabla \smooth(x)$ and $u = \oracle(x)$ and use that $\|u - v\| \leq a\|v\| + b$. The first part of the theorem is then a consequence of the triangle inequality. The second part follows since $\|u\|^2 \geq (1-a)^2\|v\|^2  - 2b (1-a)\|v\|+ b^2$ and
$$\|u\|^2 - 2\dotp{u, v} + \|v\|^2 = \|u - v\|^2 \leq a^2\|v\|^2 + 2ab\|v\| + b^2,$$
which implies the following:
 \begin{align*}
   2\dotp{u, v} &\geq (1-a)^2\|v\|^2 + (1-a^2)\|v\|^2 - 2(1 - 2a)b\|v\|\\
                & = 2(1-a)\|v\|^2 - 2(1-2a)b\|v\|\\
                & \geq 2(1-a)(1-c)\|v\|^2 - \frac{(1-2a)^2}{2(1-a)c}b^2\\
                & \geq 2(1-a)(1-c)\|v\|^2 - \frac{1}{2c}b^2
 \end{align*}
where the third inequality uses $a \leq 1/2$ and the second inequality follows from Young's inequality: $2 \cdot ((1-2a)b \cdot \|v\|) \leq ((1-2a)b)^2/(2c(1-a)) + 2c(1-a)\|v\|^2$. To complete the result, set $c = 1/8$.
 \end{proof}

As a consequence of this Lemma, we prove that the function $\smooth$ decreases along the inexact gradient descent sequences with oracle $\oracle$.
\begin{lemma}[\textbf{Descent lemma}]\label{lemma:decreasing}
Given $\x_0 \in \RR^d$, consider the inexact gradient descent sequence: $\x_{t+1} \leftarrow \x_t - \eta \cdot \oracle_t(\x_t)$.  Then for all $t \geq 0$, we have
  \begin{equation}
    \label{eq:10}
      \smooth(\x_t) - \smooth(\x_0) \leq  - \frac{\eta}{8} (1-\cc) \sum_{i=0}^{t-1}\|\nabla \smooth(\x_i)\|^2+ {5t\eta} \ec^2.
    \end{equation}
\end{lemma}
\begin{proof}
  Since the function $\smooth$ has $\Lg$-Lipschitz gradients we have
    \begin{align*}
      \smooth(\x_{t+1})
              & \leq \smooth(\x_t) - \eta \dotp{\nabla \smooth(\x_t), \oracle(\x_t)} + \frac{\Lg\eta^2}{2} \| \oracle(\x_t)\|^2 \\
              & \leq \smooth(\x_t) -  \eta\frac{7(1-\cc)}{8} \|\nabla \smooth(\x_t)\|^2 + {2\eta} \ec^2 + \frac{\Lg\eta^2}{2}\left((1+\cc) \|\nabla \smooth(\x_t)\|+\ec\right)^2\\
              & \leq \smooth(\x_t) -  \eta\frac{7(1-\cc)}{8} \|\nabla \smooth(\x_t)\|^2  + {2\eta} \ec^2\\ &\hspace{2cm} + \frac{\Lg\eta^2}{2}\left(\frac{6}{5}(1+\cc)^2 \|\nabla \smooth(\x_t)\|^2+6\ec^2\right).
    \end{align*}
 Here the second inequality follows from Lemma~\ref{lemma:facts} and the third follows from Young's inequality: $2 (1+\cc)\|\nabla \smooth (\x_t)\|b \leq \frac{1}{5}(1+\cc)^2\|\nabla \smooth(\x_t)\|^2 + 5b^2.$
Next, observe that
\begin{align*}
 &-\eta\frac{7(1-\cc)}{8} \|\nabla \smooth(\x_t)\|^2  + {2\eta} \ec^2 + \frac{\Lg\eta^2}{2}\left(\frac{6}{5}(1+\cc)^2 \|\nabla \smooth(\x_t)\|^2+6\ec^2\right) \\
 &\leq  - \eta\left(\frac{7(1-\cc)}{8} -  \frac{6}{10}(1+\cc)^2 \right) \|\nabla \smooth(\x_t)\|^2  + \left(2 + 3 \right)\eta \ec^2 \\
  &\leq  -   \frac{\eta(1-\cc)}{8}\|\nabla \smooth(\x_t)\|^2 + {5\eta}\ec^2,
\end{align*}
 where the second line follows since $\eta \leq 1/\Lg$ and the last inequality follows from $(6/10)(1+\cc)^2 \leq (3/4)(1-\cc)$ for $\cc \leq 1/20$. Thus, we have shown that
$$
g(y_t) - g(y_0) \leq - \frac{ \eta(1-\cc)}{8}\|\nabla \smooth(\x_t)\|^2 + {5\eta}\ec^2,
$$
which implies~\eqref{eq:10}.
  \end{proof}

  As a consequence of the above Lemma, we now show that inexact gradient descent sequences $\{\x_t\}$ either (a) significantly decrease $g$ or (b) remain close to $\x_0$.
  \begin{lemma}[\textbf{Improve or localize}]\label{lemma:stayClose}
Given $\x_0 \in \RR^d$, consider the inexact gradient descent sequence: $\x_{t+1} \leftarrow \x_t - \eta \cdot \oracle_t(\x_t)$. Then, for all $\tau\leq t$, we have
    \begin{equation}
      \label{eq:12}
       \|\x_\tau - \x_0\|^2\leq   16\eta t \frac{(1+\cc)^2}{(1-\cc)}\left(\smooth(\x_0)-\smooth(\x_t)+  \left({5 } + \eta \right)t\ec^2 \right).
    \end{equation}
  \end{lemma}
  \begin{proof}
    By Lemma~\ref{lemma:facts}, we have
    \begin{align*}
      \|\x_\tau - \x_0\|^2  =  \eta^2 \left\|\sum_{i=0}^{\tau-1}\oracle(\x_i) \right\|^2
      & \leq  \eta^2 \left(\sum_{i=0}^{t-1}(1+\cc)\|\nabla \smooth (\x_i)\| + t\ec\right)^2 \\
      & \leq  2 \left(t\eta^2 \sum_{i=0}^{t-1}(1+\cc)^2\|\nabla \smooth(\x_i)\|^2+ \eta^2 t^2\ec^2\right),
    \end{align*}
    where the last inequality follows from Jensen's inequality.
    Next apply Lemma~\ref{lemma:decreasing}, to bound $\eta^2\sum_{i=0}^{t-1}\|\nabla \smooth(\x_i)\|^2 \leq \frac{8\eta}{(1-\cc)} (g(\x_0 ) - g(\x_t) + 5 \ec^2)$. Plugging this bound into the above inequality, we have
    \begin{align*}
      \|\x_\tau - \x_0\|^2 & \leq  2 \left(8\eta t\frac{(1+\cc)^2}{(1-\cc)}\left(\smooth(\x_0)-\smooth(\x_t)+  {5}\ec^2 t\right) + \eta^2 t^2\ec^2\right)\\
                         & \leq  16\eta t \frac{(1+\cc)^2}{(1-\cc)}\left(\smooth(\x_0)-\smooth(\x_t)+  \left({5} + \eta \right)t\ec^2 \right).
    \end{align*}
    This concludes the proof.
  \end{proof}

 In the next two Lemmas, we show that, when randomly initialized near  a critical point with negative curvature, inexact gradient descent sequences decrease the objective $\smooth$ with high probability. The first result (Lemma~\ref{lemma:widthBound}) will help us estimate the failure probability.

  \begin{lemma} \label{lemma:widthBound}
 Fix a point $\tilde \x$ satisfying
  $\norm{\nabla \smooth(\tilde \x)} \leq \eps \text{ and } \lambda_{\min} (\nabla^2 \smooth(\tilde \x)) \leq - \epsh$ and let $e_0$ denote an eigenvector associated to the smallest eigenvalue of $\nabla^2 \smooth(\tilde \x)$. Consider two points $\x_0$ and $\x_0'$ with
  $$\x_0 = \x'_0 + \eta r_0 e_0 \quad \text{and} \quad \max\{\|\x_0 - \tilde \x\| , \|\x'_0 - \tilde \x\| \} \leq \eta r,$$
where $r_0 \geq \omega := \frac{1}{\eta}2^{3-\gamma} R$.
Let $\{\x_t\}, \{\x'_t\}$ be two inexact gradient descent sequences, initialized at $\x_0$ and $\x_0'$, respectively:
\begin{align*}
\x_{t+1} = \x_t  - \eta \oracle(\x_t)  && \text{ and } && \x_{t+1}' = \x_t' - \eta \oracle(\x_t').
\end{align*}
 Then $\min\{\smooth(\x_\Kc) - \smooth(\x_0), \smooth(\x_\Kc') - \smooth(\x_0')\} \leq -\Fc$.
\end{lemma}
\begin{proof}
  We argue by contradiction. Suppose that
  $$\max\{ \smooth(\x_0) - \smooth(\x_\Kc) ,\smooth(\x_0') -  \smooth(\x_\Kc') \} < \Fc.$$
  Then by Lemma~\ref{lemma:stayClose}, the iterates of both sequences remain close to their initializers:
  \begin{align}\label{eq:iteratesremainclosey}
    \max\{\|\x_t - \x_0\|, \|\x_t' - \x_0'\|\} &\leq \sqrt{16\eta \frac{(1+\cc)^2}{(1-\cc)} \Kc \left(\Fc+ \left(5 + \eta \right) M \ec^2\right) } \\
                                           & \leq \sqrt{32\eta \frac{(1+\cc)^2}{(1-\cc)} \Kc \Fc}, \qquad \text{ for all }t \leq M. \notag
  \end{align}
 where the second inequality follows from two upper bound: $\eta \leq 1/\Lg$ and $\ec^2 \leq \frac{\Lg\Fc}{ M(5\Lg + 1)}$.
  We now use~\eqref{eq:iteratesremainclosey} to show for all $t \leq \Kc$, iterates $\x_t$ and $\x_t'$ remain close to $\tilde \x$. By Lemma~\ref{lemma:horrible}, we get
  \begin{align}\begin{split} \label{eq:boundR}
    \max\{\norm{\x_t - \tilde \x}, \norm{\x_t' - \tilde \x}\} & \leq \max\{\|\x_t - \x_0\|, \| \x_t' - \x_0'\|\} + \max\{\|\x_0 - \tilde \x\|, \|\x_0' - \tilde \x\|\} \\
    & \leq \sqrt{32\eta \frac{(1+\cc)^2}{(1-\cc)} \Kc \Fc} + \eta r < R.
  \end{split}
  \end{align}
  In the remainder of the proof, we will argue that inequality \eqref{eq:boundR} cannot hold. In particular, we will show that negative curvature of $g$ implies the sequences $\x_t$ and $\x_t'$ must rapidly diverge from each other.

To leverage negative curvature, we first claim that $g$ is $C^2$ with $\Lh$-Lipschitz Hessian in $\BB_R(\tilde \x)$, which contains $\x_t$ and $\x_t'$ for $t \leq \Kc$. Indeed, since $\tilde \x$ satisfies $\|\nabla \smooth(\tilde \x)\| \leq \eps \leq \alpha$, Assumption~\ref{ass:localSmooth2} ensures $\nabla^2 g(\x)$ is defined and $\Lh$-Lipschitz through $B_{\beta}(\tilde \x)$. The claim then follows since $R= \frac{1}{4\gammac} {\frac{\epsh}{\Lh}} \leq \beta$, which follows from the assumption $\epsh \leq 4\gamma\beta \Lh$

Now observe that $\{\x_t' + s( \x_t-\x_t')\mid s\in[0,1]\} \subseteq \BB_{R}(\tilde \x)$ for all $t \leq \Kc$. Therefore, defining $\cH := \nabla^2 g(\tilde \x)$, $v_t := \nabla \smooth(\x_t) - \oracle(\x_t)$,  $v_t' := \nabla \smooth(\x_t') - \oracle(\x_t')$, and $\widehat{\x}_t := \x_t - \x_t'$, we have for all $t \leq \Kc -1$
  \begin{align*}
    \hat \x_{t+1} &  = \hat \x_t - \eta (\nabla \smooth (\x_{t+1}) - \nabla \smooth (\x_{t+1}')) - \eta (v_t - v_t') \\
                  &  = (I- \eta \cH)\hat \x_t - \eta \left[\int_0^1\left( \nabla^2 \smooth(\x'_t + s (\x_t - \x_t')) - \cH \right)  ds\right]\hat \x_t  - \eta (v_t - v_t')\\
                  &  = \underbrace{(I- \eta \cH)^{t+1}\hat \x_0}_{=: p(t+1)} -\underbrace{\eta \sum^t (I - \eta \cH)^{t-\tau} \left[\int_0^1\left( \nabla^2 \smooth(\x'_\tau + s (\x_\tau - \x_\tau')) - \cH \right)  ds\right]\hat \x_\tau}_{=: q(t+1)}\\
    & \hspace{3.5cm}- \underbrace{\eta \sum^t_{\tau = 0} (I-\eta \cH)^{t-\tau}(v_\tau - v_\tau')}_{=: n(t+1)}
  \end{align*}
where the last equality follows from the recursive definition of $y_t$ and $y_t'$.
In what follows we will argue that $p(t)$ diverges exponentially and dominates $q(t)$ and $n(t)$.

Beginning with exponential growth, notice that $\widehat{\x}_0$ is an eigenvector of $\cH$ with eigenvalue $\lambda :=  -\lambda_{\min}(\cH)$. Therefore,
\begin{align}\label{eq:pdef}
  \|p(t)\| = (1+\eta\lambda)^t\|\widehat{\x}_0\|= (1+\eta\lambda)^t\eta r_0.
\end{align}
Consequently, if $\max\{\|q(t)\|, 2\|n(t)\|\} \leq \frac{\|p(t)\|}{2}$, then the following bound would hold:
  \begin{align*}
    \max\{\|\x_\Kc - \tilde \x\| , \| \x_\Kc' - \tilde \x\|\} &\geq \frac{\norm{\hat \x_\Kc}}{2} \\
    &\geq \frac{1}{2}\left( \norm{p(\Kc)} - \norm{q(\Kc)} - \norm{n(\Kc)}\right) \\
    &\geq \frac{1}{8} \|p(\Kc)\|  \\
   & = \frac{(1+\eta \lambda)^\Kc \eta r_0}{8} \\
   &\geq 2^{\gamma - 3}\eta r_0 \geq \Rc,
  \end{align*}
  where the fourth inequality follows since $M = \gamma/\eta \epsh$, $(1+\eta\lambda) \geq (1+\eta \epsh)$ and $(1+x)^{1/x}\geq 2$ for all $x \in (0, 1)$, while the final inequality follows since $r_0 \geq  \omega = \frac{\Rc}{2^{\gamma-3} \eta}$. Thus, by proving the following claim, we will contradict \eqref{eq:boundR} and prove the result.

  \begin{claim}\label{cl:small_norm}
    For all $t\leq \Kc$, we have $\max\{\|q(t)\|, 2\|n(t)\|\} \leq \frac{\|p(t)\|}{2}$.
  \end{claim}
    The proof of the claim follows by induction on $t$ and the following bound
    $$
\|I - \eta \cH\| \leq (1+\eta \lambda),
    $$
    which holds since $\eta$ is small enough that $I - \eta \cH \succcurlyeq 0$.

   Turning to the inductive proof, we note that the base case holds since $$2n(0) = q(0) = 0 \leq \|\hat \x_0\|/4.$$
   Now assume the claim holds for all $\tau \leq t$. Then for all $\tau \leq t$ we have
    $$ \norm{\hat \x_\tau} \leq \norm{p(\tau)} + \norm{q(\tau)} + \norm{n(\tau)} \leq 2 \norm{p(\tau)} \leq 2 (1+ \eta \lambda)^\tau \eta r_0,$$
    where the final inequality follows from~\eqref{eq:pdef}.
    Consequently, we may bound $\|q(t+1)\|$ as follows:
    \begin{align*}
      \|q(t+1)\| &\leq \eta \sum^t_{\tau=0} \left\|I - \eta \cH\right\|^{t-\tau}  \left\|\int_0^1\left( \nabla^2 \smooth (\x'_\tau + s (\x_\tau - \x_\tau')) - \cH \right)  ds\right\| \left\|\hat \x_\tau \right\| \\
      & \leq \eta \Lh \sum^t_{\tau=0} \left\|I - \eta \cH\right\|^{t-\tau} \max\{\norm{\x_t - \tilde \x}, \norm{\x_t' - \tilde \x}\}  \left\|\hat \x_\tau \right\| \\
                 & \leq \eta \Lh R \sum^t_{\tau=0}\left\|I - \eta \cH\right\|^t \eta r_0 \\
                 & = \eta \Lh \Rc \Kc \left\|I - \eta \cH\right\|^t \eta r_0 \\
                 &\leq 2 \eta \Lh \Rc \Kc \|p(t+1)\| \\
                 &\leq \frac{\|p(t+1)\|}{2},
    \end{align*}
    where the second inequality follows from $\Lh$-Lipschitz continuity of $\nabla^2 \smooth$ on $B_R(\tilde \x)$, the third inequality follows from the inclusions $\x_t, \x_t' \in B_R(\tilde \x)$, the fourth inequality follows from~\eqref{eq:pdef}, and the fifth inequality follow from $2 \eta \Lh \Rc \Kc \leq 1/2$. This proves half of the inductive step.

  To prove the other half of the inductive step, we bound $\|n(t+1)\|$ as follows:
    \begin{align*}
      \|n(t+1)\| & \leq \eta \sum^t_{\tau = 0}  \norm{I-\eta \cH}^{t-\tau}\norm{v_\tau - v_\tau'} \\
                 & \leq \eta  \sum^t_{\tau = 0}  \norm{I-\eta \cH}^{t-\tau}\left[\cc \left(\norm{\nabla \smooth(\x_\tau)} + \norm{\nabla \smooth(\x_\tau')} \right) + 2\ec\right] \\
                 & \leq 2\eta  \sum^t_{\tau = 0}  \norm{I-\eta \cH}^{t-\tau}\Big[\cc \left( \Lg\Rc + \eps \right) + \ec\Big] \\
                 & \leq 2\eta {(1+\eta \lambda)}^{t}\Big[\Kc \cc \left( \Lg \Rc + \eps \right) + M{ \ec}\Big]
    \end{align*}
    where the third inequality follows from $\Lg$ Lipschitz continuity of $\nabla \smooth$, the inclusions $\x_t, \x_t' \in B_R(\tilde \x)$,  and the bound $\|\nabla \smooth(\tilde \x)\| \leq \eps$; and the fourth inequality follows from the bound $\norm{I-\eta \cH}^{t-\tau} \leq (1+\eta \lambda)^t$.
    To complete the proof, we recall that three inequalities:
   $
    b \leq \frac{R}{\Kc\eta 2^{(\gamma+2)}},
    $
       $\cc \leq \frac{1}{\eta M 2^{\gamma  +2}}\min\{\frac{1}{\Lg}, \frac{R}{\eps}\}$, and $r_0 \geq  \omega = \frac{\Rc}{2^{\gamma-3} \eta}$. Then, we find that
    \begin{align*}
    \|n(t+1)\| & \leq 2\eta {(1+\eta \lambda)}^{t}\Big[\Kc \cc \left( \Lg \Rc + \eps \right) + M{ \ec}\Big] \\
    &\leq \frac{3(1+\eta\lambda)^tR}{2^{\gamma+1}} \\
    &\leq  \frac{3(1+\eta\lambda)^{t}\eta r_0}{16}  \\
     &\leq \|p(t+1)\|/4.
    \end{align*}
  This concludes the proof of the claim. Consequently, the proof of the Lemma is complete.
\end{proof}

Using the Lemma~\ref{lemma:widthBound}, the following Lemma proves that inexact gradient descent will decrease the objective value by a large amount if it is randomly initialized near a point with negative curvature.

  \begin{lemma}[\textbf{Descent with negative curvature}] \label{lemma:escaping}
   Fix a point $\tilde \x$ satisfying
  $\norm{\nabla \smooth(\tilde \x)} \leq \eps \text{ and } \lambda_{\min} (\nabla^2 \smooth(\tilde \x)) \leq - \epsh$.

Consider an initial point $\x_0 := \tilde \x + \eta\cdot u$ with $u \sim \text{Unif}(r\BB)$. Let $\{\x_t\}$ be an inexact gradient descent sequence, initialized at $\x_0$:
\begin{align*}
\x_{t+1} = \x_t  - \eta \oracle(\x_t).
\end{align*}
Then with probability at least
\begin{equation}\label{eq:escaping} p := 1 -  \Lg \frac{(1+\cc)^2}{(1-\cc)} \frac{\sqrt{d}}{ \epsh} \gamma^2 \max\left\{1, 5\frac{\Lh\eps}{\Lg\epsh}\right\} 2^{9-\gamma}, \end{equation}
we have $\smooth(\x_\Kc) - \smooth(\tilde \x) \leq -\Fc/2$
\end{lemma}
\begin{proof}
We show that the bound $\smooth(\x_\Kc) - \smooth(\tilde \x) \leq -\Fc/2$ follows from the inequality $\smooth(\x_\Kc) - \smooth(\x_0) \leq -F$. To that end, first observe that.
   $$
  \smooth(\x_0) - \smooth(\tilde \x) \leq  \dotp{\nabla \smooth(\tilde \x), \x_0 - \tilde \x} + \frac{\Lg\eta^2}{2}\| \x_0 - \tilde \x\|^2 \leq \eps \eta r + \frac{\Lg\eta^2r^2}{2} \leq -F/2
  $$
  where the last inequality follows by Lemma~\ref{lemma:horrible}. Consequently,
  \begin{align*}
    \smooth(\x_\Kc) - \smooth( \tilde \x) \leq \smooth(\x_\Kc) - \smooth(\x_0) + \smooth(\x_0) - \smooth(\tilde \x)                                                                                                                       \leq -F/2.
  \end{align*}
  This shows that it is sufficient to study  $\smooth(\x_\Kc) - \smooth(\x_0) \leq -F$  as desired.

  In the remainder of the proof, we show the event $\{\smooth(\x_\Kc) - \smooth(\x_0) \leq -F\}$ holds with the claimed probability in~\eqref{eq:escaping}. To that end, given any $\x_0' \in \RR^d$, let us define $T_\Kc(\x_0') = \x_{\Kc}'$, where  $\x_{t+1}' = \x_t' - \eta \oracle(\x_t')$ for all $t \geq 0$. Consider the set of points $\x \in \BB_{\eta r}(\tilde \x)$, for which $\Kc$ steps of the inexact gradient method with oracle $\oracle$ fail to decrease the $g$ significantly:
  \begin{align*}
    \cX_{\text{stuck}} = \{\x \in \BB_{\eta r}(\tilde \x)  \mid  \smooth(T_{\Kc}(\x)) - \smooth(\x_0) > -F \}.
  \end{align*}
We now show that $P(\x_0 \in \cX_{\text{stuck}}) \leq 1- p$. Indeed, Lemma \ref{lemma:widthBound} shows that there exists $e_0 \in \mathbb{S}^{d-1}$ such that width of $\cX_{\text{stuck}}$ along $e_0$ is upper bounded by $\eta \omega$. Thus the volume of $\cX_{\text{stuck}}$ is bounded by the volume of the cylinder $[0, \omega] \times \BB_{\eta r}^{d-1}(0)$, which yields the result:
  \begin{align*}
    \PP(\x_0 \in \cX_{\text{stuck}}) = \frac{\Vol(\cX_{\text{stuck}})}{\Vol(B_{\eta r}^d(0))}
                                    &\leq \frac{\eta \omega  \cdot \Vol(\eta r \BB^{d-1})}{\Vol(\eta r\BB^d)}\\
                                    &\leq \frac{\omega  \cdot \Gamma\left(\frac{d+1}{2}+\frac{1}{2}\right)}{r \sqrt{\pi} \Gamma\left(\frac{d+1}{2}\right)}\\
                                    & \leq \frac{\omega}{r} \cdot \sqrt{\frac{d}{\pi}} \\
                                    & \leq \frac{2^{3-\gamma}R}{\eta r} \cdot \sqrt{\frac{d}{\pi}} \\
                                    &\leq \Lg \frac{(1+\cc)^2}{(1-\cc)} \frac{\sqrt{d}}{ \epsh} \gamma^2 \max\left\{1, 5\frac{\Lh \eps}{\Lg \epsh}\right\} 2^{9-\gamma}.
  \end{align*}
   where the second inequality follows from the identity $\Vol(\eta r\BB^d) =(\eta r)^d \pi^{d/2}/ \Gamma(\frac{d}{2}+1)$; the third inequality follows from the bound $ \Gamma(x + \frac{1}{2})/ \Gamma(x)\leq \sqrt{x}$ for any $x\geq 0$ \cite{jameson2013inequalities}; the fourth inequality follows from the definition $\omega = \frac{R}{2^{\gamma-3}\eta}$; and the fifth inequality follows from the definitions $\eta = (1-a)/\Lg(1+a)^2$, $\Rc= \frac{1}{4\gamma}\frac{\epsh}{\Lh}$, and $r = \frac{\epsh^2}{400\Lh\gamma^3} \min\left\{1, \frac{\Lg \epsh}{5\eps \Lh}\right\}$, as well as the bound $400 \cdot 2^{3}/(4\sqrt{\pi}) \leq 2^9$. This concludes the proof.
\end{proof}

To conclude this section, we now combine all the Lemmas to prove Theorem~\ref{thm:main}.
\begin{proof}[Proof of Theorem~\ref{thm:main}]Set the number of iterations to
  \begin{equation*} \small T = 8 \Delta_g \max \left\{\frac{\Kc}{\Fc} , \frac{256}{\eta \eps^2}\right\} + 4\Kc.
  \end{equation*}
Then, we will prove the slightly stronger claim that there is at least one $(\eps/4,\epsh)$-second-order critical point. Let $\{x_t\}^T_{t=0}$ be the sequence generated by Algorithm \ref{alg:pertInexGD}. We partition this sequence into three disjoint sets:
  \begin{enumerate}
  \item The set of $(\eps/4, \epsh)$-second-order critical points, denoted $\cS_2$.
  \item The set of $(\eps/4)$-first-order critical points that are not in $\cS_2,$ denoted $\cS_1$.
  \item All the other points $\cS_3 = \{x_t\}_{t=0}^T \setminus (\cS_1 \cup \cS_2)$.
  \end{enumerate}
  We first prove that $|\cS_3| \leq T/4$:
     \begin{align*}
      \smooth(x_T) - \smooth(x_0) &= \sum_{t=0}^{T-1}\left(\smooth(x_{t+1})-\smooth(x_t)\right) \\
                                  & \leq -\eta \frac{(1-\cc)}{8}\sum_{t=0}^{T-1}\|\nabla g(x_t)\|^2 + 5 \eta T b^2\\
                                  & \leq -\eta \frac{(1-\cc)}{8}\sum_{t\in \cS_3}\|\nabla g(x_t)\|^2 +5 \eta T b^2 \\
                                  & < -\eta |\cS_3|  \eps^2(1-\cc)\frac{1}{128} + 5 \eta T b^2
   \end{align*}
   Rearranging, and applying $\ec^2 \leq \frac{\eps^2}{4096}$, we find
    $$
   |\cS_3| \leq \frac{\smooth(x_0) -  \smooth(x_T)}{\eta \eps^2(1-\cc)\frac{1}{128}} + \frac{5  T b^2}{\eps^2(1-\cc)\frac{1}{128}} \leq \frac{T}{(1-\cc)16} + \frac{640T}{(1-a)4096} \leq T/4,
   $$
   since $ a \leq 1/20$.

Now suppose for the sake of contradiction that $|\cS_2|$ is empty. Define $\Gamma \subset [T]$ be the set of iteration numbers where Algorithm~\ref{alg:pertInexGD} adds a perturbation to the iterate:
 $$
   \Gamma : = \{ t \in [T] \mid \|G(x_t)\| \leq \eps/2 \text{ and } t - t_{\text{pert}} \geq \Kc\}.
 $$
Every $x_t$ with $t \in \Gamma$ is first-order stationary, since
   $$\|\nabla \smooth(x_t)\| \leq \frac{1}{1-a}\left(\|\oracle(x_t)\| + b\right) \leq \frac{1}{1-a}\left(\frac{\eps}{2} + b\right) \leq \frac{20}{19} \left(\frac{\eps}{2} + \frac{\eps}{64}\right) \leq {\eps}.
   $$
Moreover, since $|\cS_2|$ is empty, such $x_t$ satisfy $\lambda_{\min}(\nabla^2 g(x_t)) < - \epsh$. Therefore, by Lemma~\ref{lemma:escaping} and a union bound, the following event
   $$ \cE = \left\{\smooth (x_{t + \Kc}) - \smooth(x_t) \leq -\frac{F}{2}\quad \text{for all }t \in \Gamma\right\}$$
does not happen with probability at most
   \begin{equation}
     \label{eq:33}
   \PP(\cE^c) \leq \frac{T \Lg \frac{(1+\cc)^2}{(1-\cc)} \frac{\sqrt{d}}{ \epsh} \gamma^2 \max\left\{1, 5\frac{\Lh \eps}{\Lg \epsh}\right\} 2^{9}}{2^\gamma}.
 \end{equation}
By Lemma~\ref{lemma:horrible}, this probability is upper bounded by $\delta$. Therefore, throughout the remainder of the proof, we suppose the event $\cE$ happens. In this event we will show that we will show that $g(x_t) < \inf g$ for some $t$, which yields the desired contradiction.

To that end, recall that by Lemma~\ref{lemma:decreasing}, $g$ cannot increase by much at each iteration:
$$
   \smooth(x_{t+1}) - \smooth(x_{t}) \leq 5 \eta  b^2 \qquad\text{for all } t \in [T].
$$
Thus, defining $t_{\text{last}} := \max\{t \mid t + M < T\}$ and we find that
\begin{align*}
     \smooth(x_{t_{\text{last}}+M+1}) - \smooth(x_0) & = \sum_{t=0}^{t_{\text{last}+M}} (\smooth(x_{t+1})-\smooth(x_t))  \\
&\leq \sum_{\substack{k \in \Gamma \\ k \leq t_{\text{last}}}} \sum_{t \in [k, k + \Kc-1]} \left( \smooth(x_{t+1}) - \smooth(x_t)\right) + 5\eta b^2 |T|\\
&= \sum_{\substack{k \in \Gamma \\ k \leq t_{\text{last}}}}\left( \smooth(x_{t+\Kc}) - \smooth(x_t)\right) +  5\eta b^2 |T|\\
&\leq -(|\Gamma| - 1)F/2 + 5\eta b^2 |T|
\end{align*}
To arrive at the desired contradiction, we will show that $|\Gamma|$ is  large. In particular, we claim that
$$
|\Gamma| \geq \frac{3T}{4M}.
$$
To prove this claim, first observe that the definition of Algorithm~\ref{alg:pertInexGD} ensures that $\{x_t \mid \|G(x_t)\| \leq \eps/2\} \subseteq \bigcup_{k \in \Gamma} \{k, \ldots, k + \Kc\}$. Moreover, $\cS_1 \subseteq \{x_t \mid \|G(x_t)\| \leq \eps/2\}$ by
Lemma~\ref{lemma:facts}:
$$\|\nabla \smooth(x_t)\| \leq \eps/4 \implies \|\oracle(x)\| \leq (1+a)\frac{\eps}{4} + b\leq \frac{21}{20} \frac{\eps}{4} + \frac{\eps}{64}\leq \frac{\eps}{2},$$
since $\cc \leq 1/20$ and $\ec \leq \eps/64$.
Therefore, since $|\cS_1| = T - |\cS_3| \geq 3T/4$, we have $(3T/4) \leq |\cS_1| \leq |\Gamma| \Kc$, as desired.

Finally, we find
\begin{align*}
 &\smooth(x_{t_{\text{last}}+M+1}) - \smooth(x_0) \\
 &\leq -(|\Gamma| - 1)F/2 + 5\eta b^2 |T| \\
  &\leq -\left(\frac{3T }{4M}-1\right)\frac{F}{2} + 5\eta b^2|T| \\
  &\leq -\frac{TF}{4M} + 5\eta b^2|T| \\
  & \leq -\frac{TF}{8M} < \inf g - g(x_0),
\end{align*}
where the third inequality follows since $T \geq 4M$ and the fourth inequality follows since $b^2 \leq \frac{1}{40\eta} \frac{F}{M}$. Thus, yielding a contradiction. This completes the proof.
\end{proof}

\section{Proof of Proposition \ref{prop:quadratic}} \label{sec:proofQuadratic}
To prove Part~\ref{part:1}, recall that $\|\nabla f_{\mu}(x)\| = \mu^{-1}(x - \hat x)$, so
$$
\|x - \hat x\| \leq \mu\|\nabla f_{\mu}(x)\| \leq \mu \eps,
$$
as desired. Note that this implies $x \in \cU= B_{3\epsh/4\Lh}(\hat x)$ since $\eps \leq  \frac{\epsh}{2\Lh \mu} $.

To prove the remaining statements, we recall the following consequence of the $\Lh$-Lipschitz continuity of $\nabla^2 f_\mu$ on the ball $\BB_{\beta}(x)$ \cite[Lemma 1.2.4]{intro_lect}: for all $y \in \BB_\beta(x)$
$$
f_\mu(x) + \dotp{\nabla f_\mu(x), y-x} +\frac{1}{2} \dotp{\nabla^2f_\mu(x) (y-x), y-x} - \frac{\Lh}{6}\|y-x\|^3 \leq f_\mu(y).
$$
Since $x$ is an $(\eps,  \epsh)$-second order critical point, we may  lower bound the left hand side by a simple quadratic: letting $r = 3\epsh/2\Lh$, we have
\begin{equation}
  \label{eq:31}
 q_0(y) :=  f_\mu(x) + \dotp{ \nabla f_\mu (x), y - x} - \frac{3}{4}\epsh \|y -x\|^2 \leq f_\mu(y)  \quad \text{for all }y\in\BB_{r}(x)
\end{equation}
Now,  define the quadratic
$$
q(y) := f(\widehat x) - \frac{\mu}{2}\left(1+3\mu\epsh\right)\eps^2 +  \dotp{ \nabla f_\mu (x), y  - \widehat x} -  \frac{3\epsh}{2} \|y - \hat x\|^2
$$
We claim that $q(y) \leq q_0(y)$.

Indeed, first observe that by $\nabla f_{\mu} (x) = \mu(x - \hat x)$, we have
$$
f_\mu (x)+\dotp{\nabla f_\mu (x), y - x} = f(\widehat x) - \frac{1}{2\mu}\|x - \widehat x\|^2 +  \dotp{ \nabla f_\mu (x), y  - \widehat x}.
$$
 Next, we may recenter the quadratic up to a small error:
$$
\|y - x\|^2 \leq 2\|y - \hat x\|^2 + 2 \|x- \hat x\|^2
$$
Therefore, we have
\begin{align*}
   q_0(y)
  & = f(\widehat x) - \frac{1}{2\mu}\|x - \widehat x\|^2 +  \dotp{ \nabla f_\mu (x), y  - \widehat x} -  \frac{3\epsh}{4} \|y - x\|^2 \\
  & \geq f(\widehat x) - \frac{1}{2}\left(\mu^{-1}+3\epsh\right)\|x - \widehat x\|^2 +  \dotp{ \nabla f_\mu (x), y  - \widehat x} -  \frac{3\epsh}{2} \|y -\hat x\|^2  \geq q(y),
\end{align*}
where the third inequality follows from the bound $\|\hat x - x\|^2 \leq \mu^2 \eps^2$. This proves the claim.

We now prove the remaining parts of the claim. First, Part~\ref{part:2} follows from~\eqref{eq:31} since $\cU \subseteq \BB_r(x)$ and $q(y) \leq q_0(y) \leq f_{\mu}(y) \leq f(y)$ for all $y \in \BB_r(x)$. Second, Part~\ref{part:3} follows since $\nabla q(\hat x) = \nabla f_{\mu}(x)$. Finally Parts~\ref{part:4} and~\ref{part:5} follow by direct computation.

\section{Proof of Theorem \ref{thm:generic}} \label{sec:proofGeneric}
	By \cite[Theorem 3.7]{drusvyatskiy2016generic}, there exist disjoint open sets $\{V_1,\ldots, V_k\}$ in $\RR^d$, whose union has full measure in $\RR^d$, and such that for each $i=1,\ldots,k$, there exist finitely many smooth maps $g_1,\ldots, g_m$ satisfying
	$$(\partial f)^{-1}(v)=\{g_1(v),\ldots, g_m(v)\}\qquad \forall v\in V_i.$$
	In particular, since $g_i$ are locally Lipschitz continuous, for every $v\in V_i$, there exists a constant $\ell$ satisfying
	\begin{equation}\label{eqn:union_blah}
	(\partial f)^{-1}(\BB_{\epsilon}(v))\subset \bigcup_{j=m}^k \BB_{\ell\epsilon}(g_j(v)),
	\end{equation}
	for all small $\epsilon>0$.	Moreover, by \cite[Corollary 4.8]{drusvyatskiy2016generic} we may assume that for every point $v$ in $V_i$ and for sufficiently small $\epsilon>0$ the set $g_j(\BB_{\epsilon}(v))$ is an active manifold around $g_j(v)$ for the tilted function $f(\cdot;v)=f(\cdot)-\langle v,\cdot\rangle$. Taking into account \cite[Theorem 3.1]{Davis2019ProximalMA}, we may also assume that the Moreau envelope $f_{\mu}(\cdot;v)$ of  $f(\cdot;v)$ is $C^p$-smooth on a neighborhood of each point $g_j(v)$.

	Fix now a set $V_i$ a point $v\in V_i$. Clearly, then there exist constants $r,\beta,L_2 >0$, such that for any point $y$ with $\dist(y,(\partial f)^{-1}(v))\leq r$, the Hessian $\nabla^2 f_{\mu}(\cdot;v)$ is $L_2$-Lipschitz on the ball $\BB_{\beta}(y)$. It remains to show that for all sufficiently small $\alpha>0$, any point $y$ satisfying	 $\|\nabla f_{\mu}(y;v)\|\leq \alpha$ also satisfies $\dist(y,(\partial f)^{-1}(v))\leq r$.
	To this end, consider a point $y$ with $\|\nabla f_{\mu}(y;v)\|\leq \alpha$ for some $\alpha>0$. Note the proximal point $\hat y$ of $f_{\mu}(\cdot;v)$ at $y$ then satisfies
	$$\dist(v, \partial f(\hat y))\leq \alpha\qquad \textrm{and}\qquad \|\hat y-y\|\leq \mu\alpha.$$ Therefore we deduce, $\hat y\in (\partial f)^{-1}(\BB_{\alpha}(v))$ and $\dist(y,(\partial f)^{-1}(\BB_{\alpha}(v))\leq \mu\alpha$. Thus, using \eqref{eqn:union_blah} we deduce that for sufficiently small $\alpha>0$, we have
	$$\dist(y,(\partial f)^{-1}(v))\leq (\mu+\ell)\alpha.$$
	Choosing $\alpha<r/(\mu+\ell)$ completes the proof.

\section{Proof of Theorem \ref{thm:twoSidedOracle}} \label{sec:proofTwoSided}
The proof of the theorem is a consequence of the following Lemma.

\begin{lem}\label{thm:quadOracle}
Assume that $g\colon\RR^d\rightarrow \Rb$ is $\alpha$-strongly convex  with minimizer $x^\star$. Let $g_x \colon \RR^d \rightarrow \Rb$ be a family of \emph{convex models} satisfying Assumption \ref{ass:doubleSided}.  Let $x_0 \in \RR^d$, let $\theta > q$, and consider the following sequence:
  $$
  x_{k+1}\leftarrow \argmin_{x \in \RR^d} \left\{g_{x_k}(x) + \frac{\theta}{2}\|x-x_k\|^2\right\}
  $$
  Then
  \begin{equation}
    \label{eq:20}
    \|x_{k + 1}  - x^\star\| \leq \left(\frac{\theta+q}{\alpha +\theta} \right)^{\frac{k+1}{2}}\|x_0 - x^\star\|.
  \end{equation}
\end{lem}
\begin{proof}
By $\theta$-strong convexity and quadratic accuracy, we have
\begin{align*}
  \left(g_{x_k}(x_{k+1}) + \frac{\theta}{2}\|x_k - x_{k+1}\|^2\right) +  \frac{\theta }{2}\|x^\ast - x_{k+1}\|^2  & \leq g_{x_k} (x^\ast) + \frac{\theta}{2}\|x^\ast - x_k\|^2\\
                                                                                                                  & \leq g(x^\ast) + \frac{\theta + q}{2}\|x^\ast-x_k\|^2.
\end{align*}
From $g(x_{k+1}) \leq g_{x_k}(x_{k+1}) + \frac{\theta}{2}\|x_k - x_{k+1}\|^2$ and and the above inequality, we have
\begin{align*}
g(x_{k+1})  + \frac{\theta}{2}\|x^\ast - x_{k+1}\|^2 \leq g(x^\ast) +  \frac{\theta + q}{2}\|x^\ast-x_k\|^2
\end{align*}
Subtract $g(x^\ast)$ from both sides and use $g(x_{k+1}) - g(x^\ast) \geq \frac{\alpha}{2}\|x_{k+1} - x^\ast\|^2$ to get the result.
\end{proof}
To complete the proof notice that both the function $g(y) = f + \frac{1}{2\mu}\|y - x_0\|^2$ and the models $g_{x} = f_x + \frac{1}{2\mu}\|y - x_0\|^2$ are $\alpha = (\mu^{-1} - \rho)$-strongly convex. Therefore, Theorem~\ref{thm:twoSidedOracle} follows from an application of Lemma~\ref{thm:quadOracle}.

\end{document}